\documentclass[a4paper]{amsart}
\usepackage{cases}
\usepackage[dvipsnames]{xcolor}
\usepackage{enumerate}
\usepackage{mathtools}
\usepackage{cancel}
\usepackage[utf8]{inputenc}
\usepackage{soul}
\DeclareMathOperator*{\supp}{supp}

\DeclareMathOperator*{\dist}{dist}
\usepackage[shortlabels]{enumitem}
\usepackage{amssymb}
\usepackage{amsthm}
\usepackage{cases}
\usepackage{mathrsfs}
\newtheorem{theorem}{Theorem}[section]
\newtheorem*{problem*}{Problem}
\newtheorem*{proposition*}{Proposition}
\newtheorem{lemma}[theorem]{Lemma}
\newtheorem{proposition}[theorem]{Proposition}

\newtheorem{corollary}[theorem]{Corollary}
\theoremstyle{remark}
\newtheorem*{remark*}{Remark}
\newtheorem*{lemma*}{Lemma}
\newtheorem{remark}[theorem]{Remark}
\theoremstyle{definition}
\newtheorem{definition}[theorem]{Definition}
\newtheorem*{definition*}{Definition}
\newtheorem{example}[theorem]{Example}
\newtheorem*{example*}{Example}
\newcommand{\floor}[1]{\left\lfloor #1 \right\rfloor}

\usepackage{hyperref}
\reversemarginpar

\def\N{\mathbb{N}}
\def\W{{\mathcal W}}
\def\G{{\mathcal G}}

\begin{document}

\title[Weak greedy algorithms and Markushevich bases]{Weak greedy algorithms and  the equivalence between semi-greedy and almost greedy Markushevich bases}

\date{}

\author[M. Berasategui]{Miguel Berasategui}

\author[S. Lassalle]{Silvia Lassalle}

\address{IMAS--UBA--CONICET - Pab I,
Facultad de Cs. Exactas y Naturales, Universidad de Buenos
Aires, (1428) Buenos Aires, Argentina}
\email{mberasategui@dm.uba.ar}

\address{Departamento de Matem\'{a}tica y Ciencias, Universidad de San
Andr\'{e}s, Vito Dumas 284, (1644) Victoria, Buenos Aires,
Argentina  and IMAS--CONICET.}
\email{slassalle@udesa.edu.ar}

\thanks{This project was partially supported by CONICET PIP 0483 and ANPCyT PICT-2015-2299.
}


\begin{abstract}
We introduce and study the notion of weak semi-greedy systems -which is inspired in the concepts of semi-greedy and branch semi-greedy systems and weak thresholding sets-, and prove that in infinite dimensional Banach spaces, the notions of \textit{ semi-greedy, branch semi-greedy, weak semi-greedy, and almost greedy} Markushevich bases are all equivalent. \color{black} This completes and extends some results from \cite{Berna2019}, \cite{Dilworth2003b}, and \cite{Dilworth2012}. We also exhibit an example of a semi-greedy system that is neither almost greedy nor a Markushevich basis, showing that the Markushevich condition cannot be dropped from the equivalence result. In some cases, we obtain improved upper bounds for the corresponding constants of the systems. 
\end{abstract}

\maketitle

\section{Introduction.}\label{introduction}
Let $X$ be a Banach space over the real or complex field $\mathbb{K}$, with dual space $X'$. A sequence $(x_i)_{i}$ in $X$ is \textit{fundamental} if $X=\overline{[x_i\colon i\in \N]}$, and it is \emph{minimal} or \emph{a
minimal system} if there is a sequence of biorthogonal functionals $(x_i')_{i}\subseteq X'$ (i.e., $x_i'(x_j)=\delta_{ij}$ for every $i,j$); a sequence $(x_i')_i$ in $X'$ is \textit{total} if $x_i'(x)=0$ for every $i\in \N$ implies that $x=0$. A fundamental minimal system $(x_i)_{i}\subseteq
X$ whose sequence of biorthogonal functionals is total is a \emph{Markushevich basis} for $X$. From now on, unless otherwise stated $(x_i)_{i}\subseteq X$ denotes a fundamental minimal system for a Banach space $X$ with (unique) biorthogonal functionals $(x_i')_{i}\subset X'$, and all of our Banach spaces are infinite dimensional. Given $x\in X$, 
 $\supp{(x)}$ denotes the support of $x\in X$, that is the set $\{i\in \mathbb{N}\colon x_i'(x)\not=0\}$. A \emph{decreasing ordering} for $x$ is an injective function $\varrho_{x}\colon\N\rightarrow \N$ such that $\supp{(x)}\subseteq \varrho_{x}(\N)$, and for  all $i\le j$
$$
|x_{\varrho_{x}(i)}'(x)|\ge  |x_{\varrho_{x}(j)}'(x)|.
$$
The set of all decreasing orderings for a fixed $x\in X$ will be denoted by $D(x)$.
The \emph{greedy ordering} for $x$ is the decreasing ordering $\varrho_{x}$ with the property that if $i<j$ and $|x_{\varrho_{x}(i)}'(x)|=
|x_{\varrho_{x}(j)}'(x)|$, then $\varrho_{x}(i)< \varrho_{x}(j)$.

Note that if $(x_i)_{i}\subseteq X$ is a fundamental minimal system and $(x_i')_{i}$ is a bounded sequence, then
$(x_i)_{i}$ is bounded below (i.e., there is $r>0$ such that $\|x_i\|\ge r$ for each $i\in\N$) and
$(x_i')_{i}$ is $w^*$-null, so the greedy ordering is well defined. \\
The \emph{Thresholding Greedy Algorithm} (TGA) for a fundamental minimal system with bounded
coordinates $(x_i)_{i}\subseteq X$ gives approximations to each $x\in X$ in terms of the greedy
ordering. For $m\in\N$, the $m$-term greedy approximation to $x$ is defined as follows:
$$
\mathcal{G}_{m}(x):=\sum\limits_{i=1}^{m}x_{\rho(x,i)}'(x)x_{\rho(x,i)},
$$
where $\rho\colon X\times \N\rightarrow \N$ is the unique mapping such that for each $x\in X$, the
function $\rho(x,\cdot)$ is the greedy ordering for $x$. In this paper, $\rho$ will always denote this function. Using the convention that the sum over the empty set is zero, $\mathcal{G}_0(x)=0$. 
As usual,  given a finite subset $A$ of $\N$,  $P_A$ denotes the projection with indices in $A$, that is $P_A(x)=\sum\limits_{i\in A}x_i'(x)x_i$ and $|A|$ denotes the cardinal of $A$.

The TGA was introduced by Temlyakov in \cite{Temlyakov1998}, in the context of the trigonometric system, and extended by Konyagin and Temlyakov to general Banach spaces in \cite{Konyagin1999}, where the authors defined the concept of \emph{greedy Schauder bases} (in the context of Schauder bases, ``TGA'' refers to any algorithm that gives approximations induced by a decreasing ordering; the greedy ordering is chosen for convenience to study general minimal systems; see \cite{Konyagin2002}).

\begin{definition}\label{definitiongreedy8}
A Schauder basis $(x_i)_{i}\subseteq X$ is \textit{greedy} if there is $M>0$ such that for every $x\in X$ and each $m\in \N$,
\begin{equation}
\|x-\G_{m}(x)\|\le M\sigma_{m}(x),\nonumber
\end{equation}
where $\sigma_m(x)$ is the \emph{best $m$-term approximation error} given by
$$
\sigma_m(x)= \inf \{
\|x-y\|\colon |\supp{(y)}|\le m\}.
$$
\end{definition}
Notice that, by a  density argument, $(x_i)_{i}\subseteq X$ is a  greedy basis if and only if there is $M>0$ such that for every $x\in X$ and each $m\in \N$,
\begin{equation}
\|x-\sum_{i=1}^m x'_{\rho_x(i)}(x)x_{\rho_x(i)}\|\le M\sigma_{m}(x), \nonumber
\end{equation}
for some (or for all) $\rho_x\in D(x)$, which is the original definition in \cite{Konyagin1999}.\\

In \cite{Konyagin1999}, the authors also introduce the concept of \emph{quasi-greedy} Schauder bases, developed independently for fundamental biorthogonal systems and quasi-Banach spaces (though with somewhat different terminology) by Wojtaszczyk in \cite{Wojtaszczyk2000}. This concept was studied in several papers (see for example \cite{AA2016},  \cite{AABW2019}, \cite{Dilworth2003b}, \cite{Dilworth2003}, \cite{Dilworth2012} and \cite{Konyagin2002}).  We follow the definition from \cite{Konyagin2002}.
\begin{definition}\label{definitionquasigreedy323885} A fundamental minimal system $(x_i)_i\subseteq
X$ with bounded coordinates is \emph{quasi-greedy} if there is $M>0$ such that for all $x$
and for each $m$,
\begin{equation}
\|\mathcal{G}_{m}(x)\|\le M \|x\|. \label{quasigreedy323885}
\end{equation}
\end{definition}

It is clear from the definition that a fundamental minimal system is quasi-greedy if and only if there
is a constant $M>0$ such that for all $x$ and for each $m$,
\begin{equation}
\|x-\mathcal{G}_{m}(x)\|\le M||x||. \label{alternativeconditionquasigreedy121414}
\end{equation}
In the literature, the ``quasi-greedy constant'' of the system has been defined as the minimum $M$ for
which \eqref{quasigreedy323885} holds (see, e.g., \cite{Dilworth2003b}), the minimum $M$ for which
\eqref{alternativeconditionquasigreedy121414} holds (see, e.g., \cite{Berna2019}), or the minimum $M$
for which both hold (see, e.g.,\cite{AA2016}, \cite{Dilworth2003}, or \cite{Konyagin2002}). The differences in notation in the literature have been discussed in \cite{AA2017}, where the minimum $M$ for which \eqref{alternativeconditionquasigreedy121414} holds is called the \emph{suppression quasi-greedy} constant of the basis. We will refer to them as \emph{first quasi-greedy constant}\eqref{quasigreedy323885} and the \emph{second quasi-greedy
constant} \eqref{alternativeconditionquasigreedy121414} of the system, respectively.  \\
A notion between being greedy and quasi-greedy is that of being almost greedy. This concept was introduced by Dilworth, Kalton, Kutzarova and Temlyakov for Schauder bases \cite{Dilworth2003}, and also studied in the context of Markushevich bases in several papers (see, among others, \cite{AA2017},  \cite{Berna2019}, \cite{Dilworth2018} and \cite{Dilworth2012}).

\begin{definition}\label{definitionbiorthogonalalmostgreedy281822}
A Markushevich basis $(x_i)_{i}\subseteq X$ with bounded coordinates is \emph{almost greedy} if there
is a constant $M>0$ such that for each $x\in X$ and $m\in \N_0$,
\begin{equation}
\|x-\mathcal{G}_{m}(x)\|\le M\widetilde{\sigma}_m(x),\label{equationalmostgreedy}
\end{equation}
where
$$
\widetilde{\sigma}_m(x)= \inf\{ \|x-P_A(x)\|\colon |A|\le m\}.
$$
The minimum $M$ for which the above inequality holds is called the \emph{almost greedy constant} of the basis.
\end{definition}

Another natural weakening of the greedy notion is the concept of semi-greedy Schauder bases that was introduced in \cite{Dilworth2003b}. This concept was later extended to Markushevich bases (see e.g., \cite{Berna2019} and \cite{Dilworth2015}) and can be considered for fundamental minimal systems in general. Before we give a definition, we set
 $$
\sigma_m(x)= \inf \{
\|x-y\|\colon |\supp{(y)}|\le m \text{ and } y=P_{\supp(y)}(y) \},
$$
which extends the concept of the best-$m$-term approximation error to such systems. 

\begin{definition}\label{definitionssemigreedy278129}
A fundamental minimal system $(x_i)_{i}\subseteq X$ with bounded coordinates is \emph{semi-greedy} if there exists $M>0$ such that  for every $x\in X$ and every $m\in \N$, there is $z\in [x_i:i\in \mathcal{GS}_m(x)]$ such that
$$
\|x-z\|\le M \sigma_m(x),  \label{semigreedy278129}
$$
where $\mathcal{GS}_{m}(x):=\{ \rho(x,i)\colon1\le i\le m\}$ is the \emph{greedy set} of $x$ of cardinality $m$. Under these conditions, $z$ is called an \emph{$m$-term Chebyshev approximant} to $x$, and the minimum $M$ for which the inequality holds is called the \emph{semi-greedy constant} of the system.
\end{definition}
In semi-greedy systems, the \emph{Chebyshev Greedy Algorithm} (CGA) is used instead of the TGA. The CGA gives generally better approximations than the TGA, since the approximations to $x$ are not limited to the projections. \\ 
Weaker versions of the TGA and the CGA have been studied in several papers (see  e.g.  \cite{BBG2017},\cite{BBGHO2020}, \cite{Dilworth2018}, \cite{Dilworth2015}, \cite{Dilworth2012}, \cite{Gogyan2009} \cite{Konyagin2002} and  \cite{Temlyakov1998b}). These algorithms give approximations in terms of \emph{weak thresholding sets}, which are defined (according to \cite{Konyagin2002}) as follows:

\begin{definition}
\label{defweakthreshset485875}
Let $(x_i)_{i}\subseteq X$ be a fundamental minimal system and let $0<\tau\le 1$. Given $x\in
X$ and $m\in \N$, a set $\mathcal{W}^{\tau}(x,m)$ of cardinality $m$ is called an \emph{$m$-weak
thresholding set for $x$ with weakness parameter $\tau$} if
$$
|x_i'(x)|\ge \tau |x_j'(x)|,
$$
for all $i\in \mathcal{W}^{\tau}(x,m)$ and all $j\in\N\setminus \mathcal{W}^{\tau}(x,m)$.
In this paper, $\mathcal{W}^{\tau}(x,m)$ always denotes one of these sets. 
\end{definition}
Weak thresholding greedy algorithms give approximations to $x\in X$ by projections $P_{\mathcal{W}^{\tau}(x,m)}$, whereas weak Chebyshev greedy algorithms give instead the best approximation in terms of the vectors in $[x_i: i\in \mathcal{W}^{\tau}(x,m)]$. In this paper, we focus on the following two properties: 

\begin{definition}\label{definitionweaklamostgreedy328478}
A fundamental minimal system $(x_i)_i\subseteq X$ is \emph{weak almost greedy} with weakness parameter
$0<\tau\le 1$ (WAG($\tau$)) and constant $M$ if for every $x\in X$ and every $m\in \N$, there is a weak thresholding set $\mathcal{W}^{\tau}(x,m)$ such that
\begin{equation}
\|x-P_{\mathcal{W}^{\tau}(x,m)}(x)\|\le M\widetilde{\sigma}_{m}(x).
\label{taualmostgreedy328478}
\end{equation}
\end{definition}

\begin{definition}\label{definitionweaksemigreedy747473}
Let $(x_i)_{i}\subseteq X$ be a fundamental minimal system and $0<\tau\le 1$. The system is \emph{weak
semi-greedy with weakness parameter $\tau$} (WSG($\tau$)) and constant $M$ if for every $x\in X$ and
every $m\in \N$, there is a weak thresholding set $\mathcal{W}^{\tau}(x,m)$ and $z \in [x_i: i\in \mathcal{W}^{\tau}(x,m)]$ such that
$$
\|x-z\|\le M\sigma_{m}(x).
$$
In that case, $z$ is called an \emph{$m$-term Chebyshev $\tau$-greedy approximant} for $x$.
\end{definition}
Note that the WAG$(\tau$) concept is an extension of the almost greedy concept (Definition~\ref{definitionbiorthogonalalmostgreedy281822}) that correspons to $\tau=1$. Indeed, the greedy set $\mathcal{GS}_m(x)$ is a weak thresholding set with parameter $1$, and $\G_m(x)=P_{\mathcal{GS}_m}(x)$, so any almost greedy system is WAG($1$). Reciprocally, if $(x_i)_i$ is WAG($1$), given $x \in X$, $m\in \N$, a set $\W^{1}(x,m)$ and $\epsilon>0$, one can choose $y\in X$ with the property that $|x_i'(y)|\not=|x_j'(y)|$ for all $i\not=j$ so that $\|x-y\|\le \epsilon$ and $\W^{1}(x,m)=\mathcal{GS}_m(y)$ is the only $m$-weak thresholding set for $y$. It easily follows from this that  \eqref{taualmostgreedy328478} holds for every $x$, $m$, and \emph{every} set $\mathcal{W}^{1}(x,m)$, so in particular \eqref{equationalmostgreedy} holds. Similarly, the notion of WSG($\tau$) systems is an extension of that of semi-greedy systems, which also corresponds to the case $\tau=1$. \\
We are now in a position to describe the goal of this paper,  which is twofold. 
First, we study the relations between almost greedy and semi-greedy systems, and their relation with approximations involving weak thresholding sets. In particular, we prove that for Markushevich bases, the concepts of semi-greedy, almost greedy, WSG($\tau$) and WAG$(\tau)$ systems are all equivalent, extending results from \cite{Berna2019} and \cite{Dilworth2003b}.\\ 
Second, we focus our attention on some aspects that are significant on finite-dimensional spaces, in which the relations of their respective constants are of relevance. In particular, we answer a question from \cite{Dilworth2012}.\\
This paper is structured as follows:  In Section~\ref{sectionWAG} we study weak almost greedy systems. In Section~\ref{sectionFDSP}, we define and study a separation property that will allow us to prove our main result, Theorem~\ref{theoremsemigreedyalmostgreedy327211partI}. Section~\ref{sectionmainresults} is devoted to the study of weak semi-greedy systems. Finally, in Section~\ref{sectionfiniteandBTA} we focus on finite-dimensional spaces and extend some results from \cite{Dilworth2012}.

\section{Weak almost greedy systems.}\label{sectionWAG}
In this section, we prove that the notions of WAG($\tau$) and almost greedy systems are equivalent. Note that one of the implications is immediate by definition. Indeed, any weak thresholding set $\W^{1}(x,m)$ is also a weak thresholding set $\W^{\tau}(x,m)$ for all $0<\tau\le 1$ and almost greedy systems are  WAG($1$) systems, so they are WAG($\tau$) for all $\tau$. Moreover, for almost greedy systems, it was proven in \cite[Theorem 2.2]{Konyagin2002} that \eqref{taualmostgreedy328478} holds for \emph{every} set $\W^{\tau}(x,m)$, with $M$ only depending on $\tau$ and the almost greedy constant of the basis (while \cite[Theorem 2.2]{Konyagin2002} is stated for Schauder bases, the proof does not use the Schauder property).  \\
To prove that every WAG($\tau$) system is almost greedy, we will use the known equivalence between almost greediness and quasi-greediness plus democracy or superdemocracy - two concepts we define next. We also define the concept of hyperdemocracy,  a natural extension of both democracy notions that has its roots in \cite{Dilworth2003b}, \cite{Dilworth2003} and \cite{Wojtaszczyk2000}. 
\begin{definition}\label{definitionsuperdemocraticsequence233412}
A sequence $(x_i)_{i}\subseteq X$ is \emph{superdemocratic} if there is $K>0$ such that
\begin{equation}
\|\sum\limits_{i\in A}a_ix_i\|\le K \|\sum\limits_{j\in B}b_jx_j\|,\label{superdemocraticondition}
\end{equation}
for any pair of finite sets $A, B \subseteq \N$ with $|A|\le |B|$, and any scalars $(a_i)_{i\in A}, (b_j)_{j\in B}$ such that  $|a_i|=|b_j|$ for every $i\in A$, $j\in B$. The minimum $K$ for which the
above inequality holds is called the \emph{superdemocracy constant}  of $( x_i )_{i}$.  When \eqref{superdemocraticondition} holds with $a_i=b_j=1$ for all $i,j$, the sequence is \emph{democratic}, and the corresponding constant is the \emph{democracy constant} of $(x_i)_i$, whereas if \eqref{superdemocraticondition} holds with $|a_i|\le |b_j|$ for all $i,j$, we say that the sequence is \emph{hyperdemocratic}, and that the corresponding minimum constant is the \emph{hyperdemocracy constant} of $(x_i)_i$.
\end{definition}
From \cite[Theorem 3.3]{Dilworth2003} and its proof, we extract the following result, which characterizes almost greedy Markushevich bases, valid for real and complex Banach spaces.

\begin{theorem}\label{Theoremalmostgreedydemocraticquasigreedy} Let $(x_i)_{i}\subseteq X$ be a
Markushevich basis.
\begin{enumerate}[\upshape (a)]
\item \label{pointa} If  $(x_i)_{i}$ is almost greedy with constant $K_a$, it is quasi-greedy with second quasi-greedy constant $K_{2q}\le K_a$ and democratic with constant $K_d\le K_a$.
\item \label{pointb} If $(x_i)_{i}$ is quasi-greedy with first  quasi-greedy constant $K_{1q}$ and democratic with constant $K_d$, then  it is almost greedy with constant $K_a\le 32K_d(1+K_{1q})^4$.
\end{enumerate}
\end{theorem}
Almost greedy Markushevich bases can also be characterized as quasi-greedy and superdemocratic (see, e.g., \cite{Berna2019}, \cite{Dilworth2003b} or \cite{BBG2017}, which considers complex spaces and improves the order of the bound for $K_a$ in \ref{pointb} if democracy is replaced with superdemocracy) or quasi-greedy and hyperdemocratic. In fact, it follows at once from Theorem~\ref{Theoremalmostgreedydemocraticquasigreedy}\ref{pointb} that every quasi-greedy hyperdemocratic or superdemocratic basis is almost greedy, whereas a proof that an almost greedy system is hyperdemocratic can be obtained combining Theorem~\ref{Theoremalmostgreedydemocraticquasigreedy}\ref{pointa} with \cite[Proposition~2]{Wojtaszczyk2000} and \cite[Lemma~2.2]{Dilworth2003}, with minor modifications for complex scalars. This implication is also established in Proposition~\ref{propositionwagag394995} below, taking $\tau=1$. \\ 
Also, note that in Theorem~\ref{Theoremalmostgreedydemocraticquasigreedy}, the hypothesis that the minimal system $(x_i)_i$ is a Markushevich basis is not necessary, since an almost greedy system is clearly quasi-greedy, and a quasi-greedy system is a Markushevich basis. We give a simple proof of this fact that follows from \cite[Theorem 1]{Wojtaszczyk2000} (see also the proof of~\cite[Corollary~3.5]{AABW2019}). If $(x_i)_i$ is quasi-greedy and $x_i'(x)=0$ for all $i\in \N$, then $\mathcal{G}_m(y)=\mathcal{G}_m(y-x)$ for every $y\in X$ and every $m\in \N$. Thus,
$$
\|x\|\le \|y-x\|+\|y-\mathcal{G}_m(y)\|+\|\mathcal{G}_m(y-x)\|\le \|y-\mathcal{G}_m(y)\|+(1+M)\|y-x\|.
$$
Since the system is fundamental, for any $\epsilon>0$ there is $m\in \N$ and $y\in [x_i\colon 1\le i\le m]$
such that $\|x-y\|\le \epsilon$. Given that $y=\mathcal{G}_m(y)$, it follows that $\|x\|\le (1+M)\epsilon$. As $\epsilon$ is arbitrary, this entails that $x=0$. 

Now we prove that every WAG($\tau$) is almost greedy. The proof is based on that of \cite[Proposition 4.4]{Dilworth2012}, adapted for our purposes.

\begin{proposition}\label{propositionwagag394995} Let $0<\tau\le 1$, and let  $(x_i)_i\subseteq X$ be a WAG($\tau$) system
with constant $M$. Then, $(x_i)_{i}$ is a quasi greedy Markushevich basis with first quasi-greedy constant $K_{1q}\le (1+M)(1+M^2\tau^{-4})$, and is hyperdemocratic with constant $K_{hd}\le M^2\tau^{-2}$.
Hence, $(x_i)_i$ is almost greedy. 
\end{proposition}
\begin{proof}
To prove the hyperdemocracy condition, fix nonempty finite sets $A, B\subseteq \N$ with $|A|\le|B|$, and $(a_i)_{i\in A}, (b_j)_{j\in B}$ such that $|a_i|\le |b_j|$ for every $i,j$, and choose a set $C\subseteq \N$ so that  $|C|=|B|$ and $C\cap (A\cup B)=\emptyset$.
Assume without loss of generality that $a:=\max\limits_{i\in A}|a_i|>0$. For every $0<\epsilon<1$ we have
\begin{align}
(1-\epsilon)a\tau\|\sum\limits_{k\in C}x_k\|=&\|\sum\limits_{k\in C}(1-\epsilon)a\tau
x_k+\sum\limits_{j\in B}b_jx_j-\sum\limits_{j\in B}b_jx_j\|\nonumber\\
\le& M \widetilde{\sigma}_{|C|}(\sum\limits_{k\in C}(1-\epsilon)a\tau  x_k+\sum\limits_{j\in
B}b_jx_j)\nonumber\\
\le&M\|\sum\limits_{j\in B}b_jx_j\|,\label{rightside394995b}
\end{align}
the first inequality resulting from the fact that $B$ is the only $|C|$-weak thresholding set for
$\sum\limits_{k\in C}(1-\epsilon)a\tau x_k+\sum\limits_{j\in B}b_jx_j$ with
weakness parameter $\tau$. Similarly, $C$ is the only $|C|$-weak thresholding set with weakness parameter $\tau$ for $\sum\limits_{k\in C}(1+\epsilon)a\tau^{-1} x_k+\sum\limits_{i\in
A}a_ix_i$, thus
\begin{align}
\|\sum\limits_{i\in A}a_ix_i\|=&\|\sum\limits_{i\in A}a_ix_i+\sum\limits_{k\in
C}(1+\epsilon)a\tau^{-1} x_k-\sum\limits_{k\in C}(1+\epsilon)a\tau^{-1} x_k\|\nonumber\\
\le& M\widetilde{\sigma}_{|C|}(\sum\limits_{k\in C}(1+\epsilon)a\tau^{-1} x_k+\sum\limits_{i\in
A}a_ix_i)\nonumber\\
\le&(1+\epsilon)a\tau^{-1}M\|\sum\limits_{k\in C}x_k\|.\label{leftside394995b}
\end{align}
By letting $\epsilon\rightarrow 0$, it follows from \eqref{rightside394995b} and \eqref{leftside394995b}
that
$$
\|\sum\limits_{i\in A}a_ix_i\|\le M^2\tau^{-2}\|\sum\limits_{j\in B}b_jx_j\|.
$$
Therefore, $(x_i)_i$ is hyperdemocratic with $K_{hd}\le M^2\tau^{-2}$.\\
To prove that $(x_i)_i$ is quasi-greedy, first note that for every $x\in X$, there is a weak thresholding set $\mathcal{W}^{\tau}(x,1)$, so
$(|x_i'(x)|)_{i}$ is bounded. It follows by uniform boundedness that $(x_i')_{i}$ is bounded. Since $(x_i)_i$ is fundamental, this entails that $(x_i')_i$ is $w^*$-null. Now fix $x\in X$ and $m\in \N$. If $\mathcal{G}_{m}(x)=0$, there is nothing to prove. Else, let
$$
n:=\max\{ 1\le i\le m\colon  x_{\rho(x,i)}'(x)\not=0\}.
$$
Given that $x_i'(x)\not=0$ for all $i\in \mathcal{GS}_{n}(x)$ and $(x_i')_i$ is $w^*$-null, there is $j_0\in \N$ such that for each $j\ge j_0$ and each $i\in \mathcal{GS}_{n}(x)$, 
$$
|x_j'(x)|<\tau |x_i'(x)|.
$$
Thus, if $j\ge j_0$, every weak thresholding set $\mathcal{W}^{\tau}(x,j)$ contains
$\mathcal{GS}_{n}(x)$. Hence, there is a minimum $m_1\in \N$ for which there is a weak thresholding set $\mathcal{W}^{\tau}(x,m_1)$ containing $\mathcal{GS}_{n}(x)$ and such that
\eqref{taualmostgreedy328478} holds. Now if $\mathcal{W}^{\tau}(x,m_1)=\mathcal{GS}_n(x)$, then 
\begin{align*}
\|\G_{m}(x)\|=&\|\G_{n}(x)\|\le \|x\|+\|x-\G_n(x)\|=\|x\|+\|x-\sum\limits_{i\in \mathcal{W}^{\tau}(x,m_1)}x_i'(x)x_i\|\\
\le& (1+M)\|x\|\le (1+M)(1+M^2\tau^{-4})\|x\|.
\end{align*}
On the other hand, if $\mathcal{GS}_n(x)\subsetneq \mathcal{W}^{\tau}(x,m_1)$, let $\mathcal{W}^{\tau}(x,m_1-1)$ be a weak thresholding set for which
\eqref{taualmostgreedy328478} holds. By the minimality of $m_1$ we get that
$\mathcal{GS}_{n}(x)\not\subseteq \mathcal{W}^{\tau}(x,m_1-1)$, so for every $i\in
\mathcal{W}^{\tau}(x,m_1-1)$ we have
\begin{equation}
|x_i'(x)|\ge \tau |x_{\rho(x,n)}'(x)|.\label{previousset394995b}
\end{equation}
Thus, if there is $i_0\in \mathcal{W}^{\tau}(x,m_1-1)\setminus \mathcal{W}^{\tau}(x,m_1)$,
it follows from \eqref{previousset394995b} and Definition~\ref{defweakthreshset485875} that for all
$j\in \mathcal{W}^{\tau}(x,m_1)$,
\begin{equation*}
|x_j'(x)|\ge \tau |x_{i_0}'(x)|\ge \tau^{2}|x_{\rho(x,n)}'(x)|.
\end{equation*}
On the other hand, if $\mathcal{W}^{\tau}(x,m_1-1)\subseteq \mathcal{W}^{\tau}(x,m_1)$, given that
$$
\mathcal{GS}_{n}(x)\not \subseteq \mathcal{W}^{\tau}(x,m_1-1)\quad  \text{ and }\quad
\mathcal{GS}_{n}(x)\subseteq \mathcal{W}^{\tau}(x,m_1),
$$		
it follows that there is $1\le i_1\le n$ such that
$$
\mathcal{W}^{\tau}(x,m_1)=\mathcal{W}^{\tau}(x,m_1-1)\cup \{ \rho(x,i_1)\},
$$
which implies that \eqref{previousset394995b} also holds for all $i\in \mathcal{W}^{\tau}(x,m_1)$. Therefore, in any case, we have
$$
|x_j'(x)|\ge  \tau^{2}|x_{\rho(x,n)}'(x)|
$$
for all $j\in \mathcal{W}^{\tau}(x,m_1)$. In
what follows, put $\W= \mathcal{W}^{\tau}(x,m_1)$. As for all $i\in \N\setminus
\mathcal{GS}_{n}(x)$,
$$
|x_{\rho(x,n)}'(x)|\ge |x_i'(x)|,
$$
we obtain
$$
\max_{i\in \W \setminus \G_{n}(x)}|x_i'(x)|\le \min_{j\in \W} \tau^{-2}|x_j'(x)|.
$$
Hence, using that $\G_{m}(x)=\G_{n}(x)$ and  applying the hyperdemocracy condition, we get
\begin{align*}
\|\G_{m}(x)\|  \le & \|\sum\limits_{i\in \W}x_i'(x)x_i\|+\|\sum\limits_{i\in \W}x_i'(x)x_i -
\G_{n}(x)\|\\
=&\|\sum\limits_{i\in \W}x_i'(x)x_i\|+\|\sum\limits_{i\in \W\setminus
\mathcal{GS}_{n}(x)}x_i'(x)x_i\|\\
\le&\|\sum\limits_{i\in \W}x_i'(x)x_i\|+K_{hd}\|\sum\limits_{i\in \W} \tau^{-2} x_{i}'(x)x_i\|\\
\le&(1+ M^2\tau^{-4})\|\sum\limits_{i\in \W}x_{i}'(x)x_i\|\\
\le& (1+M^2\tau^{-4})(\|x\|+\|x-\sum\limits_{i\in \W}x_{i}'(x)x_i\|)\\
\le& (1+M^2\tau^{-4})(\|x\|+M\widetilde{\sigma}_{m_1}(x))\\
\le& (1+M)(1+M^2\tau^{-4})\|x\|.
\end{align*}
This proves that $(x_i)_i\subseteq X$ is quasi-greedy (with $K_{1q}$ as in the statement). Then
$(x_i)_i$ is a Markushevich basis, and it is almost greedy by
Theorem~\ref{Theoremalmostgreedydemocraticquasigreedy}.
\end{proof}

\section{The finite dimensional separation property.}\label{sectionFDSP}
In this section, we introduce and study a separation property inspired by some of the proofs in \cite{AlKa2016}.  We give upper bounds for a constant associated with this property. The constant plays a key role in some of our proofs involving Markushevich bases. 

\begin{definition}\label{definitionseparation}Let $(u_i)_i\subseteq X$ be a sequence. We say that $(u_i)_i$ has the \emph{finite dimensional separation property} (FDSP) if there is a positive constant $M$ such that for every separable subspace $Z\subset X$ and every $\epsilon>0$, there is a basic subsequence $(u_{i_k})_k$ with basis constant no greater than $M+\epsilon$ satisfying the following: For every finite dimensional subspace $F\subset Z$ there is $j_F\in \N$ such that 
\begin{equation}
\|x\|\le (M+\epsilon)\|x+z\|,\label{separation}
\end{equation}
for all $x\in F$ and all $z\in \overline{[u_{i_k}\colon k> j_F]}$. We call any such subsequence a \emph{finite dimensional separating subsequence} for $(Z, M, \epsilon)$ (and for $(u_i)_i$, leaving that implicit when clear), and we call the minimum $M$ for which this property holds the \emph{finite dimensional separation constant} $M_{fs}$ of $(u_i)_i$. 
\end{definition}
\begin{remark}Note that in order to check whether a subsequence is finite dimensional separating for $(Z, M, \epsilon)$, it is enough to check that  \eqref{separation} holds for $x$ with $\|x\|=1$ and $z\in [u_{i_k}\colon k> j_F]$. 
\end{remark}

The following lemma gives a basic characterization for finite dimensional separating sequences. The proof is immediate and is left to the reader.

\begin{lemma}\label{lemmaordering} Let $(u_i)_i\subseteq X$ be a sequence. For any sequence $(a_i)_i\subset \mathbb{K}$ with $a_i\not=0$ for all $i$, and any biyection $\pi\colon \N \rightarrow \N$, 
$(u_i)_i$ has the finite dimensional separation property with constant $M_{fs}$ if and only if, for any $l\in \N$, $(a_i u_{\pi(i)})_{i\ge l}$ does. 
\end{lemma}

For our next result, we need the following technical lemmas; the second one is a  variant of \cite[Theorem 1.5.2]{AlKa2016}.

\begin{lemma}\cite[Lemma 1.5.1]{AlKa2016}\label{lemmabasicsubsequence219224} Let $S\subseteq X'$ be a subset such that $S$ is bounded below and $0\in \overline{S}^{w^{*}}$. Then, for every $\epsilon>0$ and every finite
dimensional subspace $F\subseteq X'$, there is $x'\in S$ such that for all $y'\in F$ and $b \in \mathbb{K}$, 
$$
\|y'\|\le (1+\epsilon)\|y'+b x'\|.
$$
\end{lemma}

\vskip .1cm

\begin{lemma}\label{lemmabasic238382}
Let $X$ be a Banach space and $(u_i')_{i}\subseteq  X'$ a sequence such that $(u_i')_{i}$ is bounded below and  $0\in \overline{\{ u_i' \}}^{w^*}_{i\in\N}$.  Then, for any separable subspace $Z\subset X'$ and $\epsilon>0$ there is a basic subsequence $(u_{i_n}' )_{n}$ with basis constant no greater than $(1+\epsilon)$ satisfying the following: For any finite dimensional subspace $F\subset Z$ and every $\xi>0$, there is $j_{F,\xi}\in \N$ such that for all $y'\in F$ and $v'\in \overline{[u_{i_n}'\colon n> j_{F,\xi}]}$, 
$$
\|y'\|\le (1+\xi)\|y'+v' \|.
$$
In particular, $(u_i')_i$ has the finite dimensional separation property with constant $1$. 
\end{lemma}

\begin{proof}
Choose a dense sequence $(v_i')_{i}$ in $Z$ and a sequence of positive scalars $(\epsilon_i)_{i}$ so that $\prod\limits_{i=1}^{\infty}(1+\epsilon_i)\le (1+\epsilon)$. Since $0\in \overline{\{ u_i'
\}}^{w^*}_{i\in\N}$ and $u_i'\not=0$ for every $i\in \N$, it follows that $0\in \overline{\{ u_i'
\}}^{w^*}_{i\ge l}$ for every $l\in \N$. Thus, by Lemma~\ref{lemmabasicsubsequence219224}, we can
choose $i_2>i_1:=1$ so that for all $y'\in [v_{i_1}', u_{i_1}']$ and all $b\in \mathbb{K}$,
$$
\|y'\|\le (1+\epsilon_1)\|y'+b  u_{i_2}' \|.
$$
By an inductive argument, we obtain a strictly increasing sequence of positive integers $\{
i_n\}_{n\in\N}$ such that for all  $y'\in [v_s', u_{s}'\colon 1\le s\le i_n]$, $b\in \mathbb{K}$
and $j\in\N$,
$$
\|y'\|\le (1+\epsilon_n)\|y'+b  u_{i_{n+1}}'\|.
$$
Then, for any  positive integers $j< l$, any $y'\in [v_s', u_s'\colon 1\le s\le i_{j}]$ and any scalars $(a_n)_{j <n\le l}$,
\begin{equation}
\|y'\|\le \prod\limits_{n=j}^{l-1}(1+\epsilon_{n})\|y'+\sum\limits_{n=j+1}^{l}a_n u_{i_n}'\|
\le \prod\limits_{n=j}^{\infty}(1+\epsilon_{n}) \|y'+\sum\limits_{n=j+1}^{l}a_n
u_{i_n}'\|.\nonumber
\end{equation}
In particular, $(u_{i_n}')_{n}$ is basic with basis constant no greater than
$\prod\limits_{n=1}^{\infty}(1+\epsilon_n)\le 1+\epsilon$, and the result holds for $F \subset [ v_i'\colon 1\le i\le n]$ for some $n\in \N$. Now, standard density arguments allow us to obtain the result for any finite dimensional subspace of $Z=\overline{[v_i'\colon i\in\N]}$.
\end{proof}

\begin{remark}\label{remarkbasic323737} Suppose that there is  a sequence $(u_i)_{i}\subseteq X$ such that $(u_i)_{i}$ is bounded below and $0\in \overline{\{ u_i \}_{i\in\N}}^{w}$. Via the canonical injection $X\hookrightarrow X''$, Lemma~\ref{lemmabasic238382} remains valid for any separable
subspace $Z\subseteq X$ where the basic sequence $(u_{i_n})_n$ is a subsequence of $(u_i)_{i}$.
\end{remark}

Next, we consider the case in which $0$ may not be a weak or a weak star accumulation point of the sequence. We use $\widehat x$ to denote $x\in X$ as an element of the bidual space $X''$, and we use $\widehat X$ to denote $X$ as a subspace of $X''$. Also, for a bounded sequence  $(u_i)_{i}\subseteq  X$, we  will consider $\beta((u_i)):=\overline{\{\widehat{u}_i\}}^{w*}_{i\in\N}\setminus \widehat{X}$.

\begin{lemma} \label{lemmanewbasic}
Let $X$ be a Banach space and $(u_i)_{i}\subseteq  X$ a bounded sequence such that  $\overline{\{u_i\}}^{w}_{i\in\N}$ is not weakly compact. Then, $(u_i)_i$ has the finite dimensional separation property with constant $M_{fs}\le M$, where
\begin{equation}
M:=\Big(2+\inf_{x''\in \beta((u_i))} \big\{ \frac{\|x''\|}{\dist{(x'',\widehat{X})}}\big\}\Big)^2.\nonumber
\end{equation}

 \end{lemma}
\begin{proof}
Since $\overline{\{u_i\}}^{w}_{i\in\N}$ is not weakly compact but  $\overline{\{\widehat{u}_i\}}^{w*}_{i\in\N}$ is weak star compact, there is $x''\in \beta((u_i))= \overline{\{\widehat{u}_i\}}^{w*}_{i\in\N}\setminus \widehat{X}$, so $M$ is well defined. Given $\epsilon>0$ and $Z\subset X$ a separable subspace, choose $0<\xi<1$ and $x_0''\in  \overline{\{\widehat{u}_i\}}^{w*}_{i\in\N}\setminus \widehat{X}$ so that
\begin{equation}
M+\xi+\xi^2+2\xi (M+\xi)\le \frac{M+\epsilon}{1+\xi},\label{xi}
\end{equation}
and 
\begin{equation}
(2+\frac{\|x_0''\|}{\dist{(x''_0,\widehat{X})}})^2\le M+\xi.\label{closetoM}
\end{equation}
Let $Z_1:=\widehat{Z}+\overline{[x_0'', \widehat{u_i}\colon i\in \N]}$, and consider in $Z_1$ the seminormalized sequence $(\widehat{u}_i-x_0'')_i$.  Then, there exists a 
basic subsequence $(\widehat{u}_{i_k}-x_0'')_{k}$ with basic constant no greater that $(1 + \xi$) satisfying the conclusions of  Lemma~\ref{lemmabasic238382}.  
Since $x_0''\not \in \widehat{X}$, there is a bounded linear functional $x_1'''$ on $\widehat{X}\oplus [x_0'']$ such that for all $x\in X$ and all $b\in \mathbb{K}$,  
$$
x_1'''(\widehat{x}+b x_0'')=b. 
$$
Suppose that $\|\widehat{x}+b x_0''\|=1$ with $b\not=0$.  Then, 
\begin{align*}
|x_1'''(\widehat{x}+b x_0'')|=&|b|=|b|\frac{\|b^{-1}\widehat{x}+x_0''\|}{\|b^{-1}\widehat{x}+x_0''\|}=\frac{1}{\|b^{-1}\widehat{x}+x_0''\|}\le\frac{1}{\dist{(x_0'',\widehat{X})}}.
\end{align*}
It follows that
\begin{equation}
\|x_1'''\|\le  \frac{1}{\dist{(x_0'',\widehat{X})}}.\label{firstbound}
\end{equation}
By the Hahn--Banach Theorem, there is a norm-preserving extension of $x_1'''$ to $X''$, which we also call $x_1'''$.  
Let $F_1:=[x_0'']$. By the choice of $(\widehat{u}_{i_k}-x_0'')_k$, there exists $j_{F_1,\xi}\in\N$ such that for all  $z''\in \overline{[\widehat{u}_{i_k}-x_0''\colon k> j_{F_1,\xi}]}$  we have
\begin{equation}
\|x_0''\|\le (1+\xi) \|x_0''+z''\|. \label{far1}
\end{equation}
In particular, given that $x_0''\not=0$ this implies that $x_0''\not \in \overline{[\widehat{u}_{i_k}-x_0''\colon k> j_{F_1,\xi}]}$. Thus, there is a bounded linear functional $x_2'''$ on $\overline{[\widehat{u}_{i_k}-x_0''\colon k> j_{F_1,\xi}]}\oplus [x_0'']$ defined, for all $z''\in \overline{[\widehat{u}_{i_k}-x_0''\colon k> j_{F_1,\xi}]}$ and all  $b \in \mathbb{K}$, by  
$$
x_2'''(z''+b x_0'')=b. 
$$
As before, for any $z''\in \overline{[\widehat{u}_{i_k}-x_0''\colon k> j_{F_1,\xi}]}$ and $b\not=0$ such that  $\|z''+b x_0''\|=1$, we have
\begin{align*}
|x_2'''(z''+b x_0'')|=&|b| =\frac{1}{\|b^{-1}z''+x_0''\|}\le \frac{1+\xi}{\|x_0''\|},&\text{by }\eqref{far1}.
\end{align*}
Thus, 
\begin{equation}
\|x_2'''\|\le  \frac{1+\xi}{\|x''_0\|}.\label{secondbound}
\end{equation}
Again, by the Hahn--Banach Theorem, we may consider $x_2'''$ defined on $X''$. Now define, for $x'' \in X''$,  the following bounded linear operators: 
\begin{align*}
T(x''):=&x''+(x_2'''-x_1''')(x'')x_0'';\\
L(x''):=&x''-(x_2'''-x_1''')(x'')x_0''.
\end{align*}
By \eqref{firstbound} and \eqref{secondbound} we get 
$$
\|T(x'')\|\le \|x''\|+\|x_2'''\|\|x''\|\|x_0''\|+\|x_1'''\|\|x''\|\|x_0''\|\le (2+\xi+\frac{\|x_0''\|}{\dist{(x_0'',\widehat{X})}})\|x''\|, 
$$
Hence, $\displaystyle \|T\|\le 2+\xi+\frac{\|x_0''\| }{\dist{(x_0'',\widehat{X})}}$ and the same bound holds for $L$. 

From this, \eqref{xi} and \eqref{closetoM} we get
\begin{align}
\|T\|\|L\|\le&  \big( 2+\xi+\frac{\|x_0''\| }{\dist{(x_0'',\widehat{X})}}\big)^2=\big(2+\frac{\|x_0''\| }{\dist{(x_0'',\widehat{X})}}\big)^2+\xi^2+2\xi \big(2+\frac{\|x_0''\| }{\dist{(x_0'',\widehat{X})}}\nonumber\\
\le& M+\xi +\xi^2+2\xi (M+\xi) \le \frac{M+\epsilon}{1+\xi}\label{thirdbound}.
\end{align}
It is easy to check that $T$ and $L$ are inverses of each other. Then, since  $T(\widehat{u}_{i_k}-x_0'')=\widehat{u}_{i_k}$ for all $k> j_{F_1,\xi}$ and  $(\widehat{u}_{i_k}-x_0'')_{k}$ is a basic sequence with basis constant no greater than $(1+\xi)$, it follows that $(u_{i_k})_{k>  j_{F_1,\xi}}$ is a basic sequence with basis constant 
$$
K_b((u_{i_k})_{k>  j_{F_1,\xi}})\le \|T\|\|L\|(1+\xi)\le M+\epsilon, 
$$
where the last inequality follows from \eqref{thirdbound}. Now let $F\subset Z$ be a finite dimensional subspace, and let $j_F:=\max{\{j_{L(\widehat{F}),\xi},j_{F_1,\xi}\}}$.  For any $x\in F$ and scalars $(a_k)_{j_F< k\le m}$, by the choice of $(\widehat{u}_{i_k}-x_0'')$ we have 
\begin{align*}
\|x\|=&\|\widehat{x}\|\le \|T\|\|L(\widehat{x})\|\le \|T\|(1+\xi)\|L(\widehat{x})+\sum\limits_{j_F< k\le m}a_k (\widehat{u}_{i_k}-x_0'')\|\\
=&\|T\|(1+\xi)\|L(\widehat{x}+\sum\limits_{j_F< k\le m}a_k \widehat{u}_{i_k})\|\le \|T\|\|L\|(1+\xi)\|\widehat{x}+\sum\limits_{j_F < k\le m}a_k \widehat{u}_{i_k})\|\\
\le& (M+\epsilon) \|x +\sum\limits_{j_F< k\le m}a_k u_{i_k}\|, 
\end{align*}
where we apply again \eqref{thirdbound} to obtain the last inequality. By a density argument, it follows that $(u_{i_k})_{k> j_{F_1,\xi}}$ has the desired properties. 
\end{proof}
\begin{remark}
Note that the upper bound for $M_{fs}$ given by Lemma~\ref{lemmanewbasic} remains unchanged if one replaces $(u_i)_i$ with $(a_iu_i)_i$, for any seminormalized sequence $(a_i)_i$.
\end{remark}
\begin{corollary}\label{corollarywhenFDSP}Let $(u_i)_i$ be a seminormalized sequence. The following are equivalent: 
\begin{enumerate}[\upshape (i)]
\item \label{basicsub} $(u_i)_i$ has a basic subsequence. 
\item \label{notcompactnozero} Either $0\in \overline{\{u_i\}}^{w}_{i\in\N}$, or $\overline{\{ u_i\}}^{w}_{i\in\N}$ is not weakly compact. 
\item \label{FDSP} $(u_i)_i$ has the finite dimensional separation property. 
\end{enumerate}
\end{corollary}
\begin{proof}
The equivalence \ref{basicsub} $\Longleftrightarrow$ \ref{notcompactnozero} was proven in~\cite{KP1965} (see also \cite[Theorem 1.5.6]{AlKa2016}).\\
By Remark~\ref{remarkbasic323737} and Lemma~\ref{lemmanewbasic} it follows that \ref{notcompactnozero} $\Longrightarrow$ \ref{FDSP}. 
Finally, \ref{FDSP} $\Longrightarrow$ \ref{basicsub} is clear. 
\end{proof}

Next we study the finite dimensional separation property of Markushevich bases, and give upper bounds for its constant in this context. Recall that for $0<c\le 1$, a subspace $S\subset X'$ is said to be $c$-norming for $X$ if 
$$
c\|x\|\le \sup_{\substack{x'\in S\\ \|x'\|=1}}|x'(x)|\quad\forall x\in X. 
$$
We will use the following result. 
\begin{lemma}\label{lemmanorming}\cite[Proposition	3.2.3]{AlKa2016}
Let $(x_i)_i$ be a Schauder basis for $X$ with basis constant $K_b$. Then $\overline{[x_i'\colon i\in \N]}\subset X'$ is $K_b^{-1}$-norming for $X$. 
\end{lemma}
Also recall that a sequence $(v_i)_i$ is a \emph{block basis} of a Markushevich basis $(x_k)_k$ if there are  sequences of positive integers $(n_i)_i$, $(m_i)_i$ with $n_i\le m_i<n_{i+1}$ for all $i$ and scalars $(b_k)_k$ such that
$$
v_i=\sum\limits_{k=n_i}^{m_i}b_kx_k,
$$
with at least one nonzero $b_k$ for each $i\in \N$. In particular, any subsequence of a Markushevich basis is a block basis of it. 

\begin{proposition}\label{propositionseparation} Let $(v_i)_i\subset X$ be a  block basis of a Markushevich basis $(y_k)_k$ for $Y\subset X$ with biorthogonal functionals $(y_k')_k$, and let $(a_i)_i$ be a scalar sequence such that $(z_i:=a_iv_i)_i$ is seminormalized. The following hold: 
\begin{enumerate}[\upshape (a)]
\item \label{Markushevich} Either $\overline{\{ z_i \}}^{w}_{i\in\N}$ is not weakly compact, or $(z_i)_i$ is weakly null. Hence, $(v_i)$ has the finite dimensional separation property, with the same constant as $(z_i)_i$.
\item \label{weakzero} If either $0\in \overline{\{z_i \}}^{w}_{i\in\N}$ or $X$ is a dual space and $0\in \overline{\{ z_i \}}^{w*}_{i\in\N}$, then $M_{fs}=1$. 
\item \label{noncompact} If $\overline{\{z_i \}}^{w}_{i\in\N}$ is not weakly compact, then 
\begin{equation}
M_{fs} \le \Big( 2+\inf \Big\{ \frac{\|x''\|}{\dist{(x'',\widehat{X})}}\colon x''\in \overline{\{\widehat{z}_i\}}^{w*}_{i\in\N}\setminus \widehat{X}\Big\}\Big)^2.\nonumber
\end{equation}
\item \label{norming} If $Y=X$ and $\overline{[y'_k \colon k\in \N]}$ is $c$-norming, then $M_{fs}\le c^{-1}$.  

\item \label{schauder} If $Y=X$ and $(y_k)_k$ is a Schauder basis for $X$ with constant $K_b$, then $M_{fs}\le K_{b}$. 
\end{enumerate}
\end{proposition}

\begin{proof}
To prove \ref{Markushevich}, suppose that $\overline{\{ z_i \}}^{w}_{i\in\N}$ is weakly compact. Then, since $Y$ is weakly closed, given a subnet $(z_{i_\lambda})$ there is a further subnet $(z_{i_{\lambda_\theta}})$ and $v_0\in Y$ such that
$$
z_{i_{\lambda_\theta}}\xrightarrow{w}v_0.
$$
Since $(z_i)_i$ is a block basis of $(y_k)_k$, it follows that $
y_k'(v_0)=0$ for all $k\in \N$, so $v_0=0$. Thus, $(z_i)_i$ is weakly null. It follows by Corollary~\ref{corollarywhenFDSP} that $(z_i)_i$ has the finite dimensional separation property, and by Lemma~\ref{lemmaordering}, so does $(v_i)_i$, with the same constant. \\
Lemma \ref{lemmabasic238382} and Remark~\ref{remarkbasic323737} imply \ref{weakzero}, and \ref{noncompact} follows by Lemma~\ref{lemmanewbasic}.\\
To prove \ref{norming}, note that by \ref{Markushevich}, $(v_i)_i$ has a basic subsequence $(v_{i_l})_l$. Let $F\subset X$ be a finite dimensional subspace, and fix $0<\epsilon<1$. Choose $0<\xi<1$ so that
$$
0<\frac{c^{-1}(1-\xi)^{-1}}{1-c^{-1}(1-\xi)^{-1}\xi}\le c^{-1}+\epsilon
$$
Take $(u_j)_{1\le j\le m_1}$ unit vectors  in $F$ that form a $\xi$-net of the unit sphere of $F$.\\
As $\overline{[y_k'\colon k\in \N]}$ is $c$-norming  so is  $[y_k'\colon k\in \N]$. Hence, there is $m_2\in \N$ and  unit vectors $(u_j')_{1\le j\le m_1}\subset [y_k'\colon 1\le k\le m_2]$ such that $|u_j'(u_j)|\ge c(1-\xi)$ for all $1\le j\le m_1$.\\
Now fix $x\in F$ with $\|x\|=1$ and $v\in \overline{[v_{i_l}\colon l >m_2]}$, and choose $1\le j\le m_1$ so that $\|x-u_j\|\le \xi$. Note that $v\in  \overline{[y_{k}\colon  k>m_2]}$, so $y_k'(v)=0$ for all $1\le k\le m_2$. Hence, 
\begin{align*}
1\le& c^{-1}(1-\xi)^{-1}|u_j'(u_j)|=c^{-1}(1-\xi)^{-1}|u_j'(u_j+v)|\le c^{-1}(1-\xi)^{-1}\|u_j+v\|\\
\le&  c^{-1}(1-\xi)^{-1}\|x+v\|+c^{-1}(1-\xi)^{-1}\|u_j-x\|\\
\le&  c^{-1}(1-\xi)^{-1}\|x+v\|+c^{-1}(1-\xi)^{-1}\xi.
\end{align*}
Thus,  
$$
\|x\|=1\le \frac{c^{-1}(1-\xi)^{-1}}{1-c^{-1}(1-\xi)^{-1}\xi}\|x+v\|\le (c^{-1}+\epsilon)\|x+v\|. 
$$
Finally, \ref{schauder} follows by \ref{norming} and Lemma~\ref{lemmanorming} 
\end{proof}

\section{Weak semi-greedy systems.}\label{sectionmainresults}

In this section, we prove our main results for weak semi-greedy minimal systems. \\
It was proven in \cite[Theorem 3.2]{Dilworth2003b} that every almost greedy Schauder basis is semi-greedy with constant only depending on its democracy and quasi-greedy constants (and thus, by Theorem~\ref{Theoremalmostgreedydemocraticquasigreedy}, only on its almost greedy constant), with a proof valid also for Markushevich bases (see also \cite[Theorem 1.10]{Berna2019} and~\cite[Corollary 4.2]{Dilworth2015}).  Moreover, it is known that if $(x_i)_i$ is an almost greedy Markushevich basis, then for each $0<\tau \le 1$ there is a constant $M$ depending only on the first quasi-greedy and the democracy constants of the basis and $\tau$  such that the conditions of Definition~\ref{definitionweaksemigreedy747473} hold for all $x\in X$, $m\in \N$, and \emph{every} weak thresholding set $\W^{\tau}(x,m)$. This fact was established in \cite[Theorem 7.1]{Dilworth2012} for finite dimensional Banach spaces, and the proof holds for the infinite dimensional case as well (see \cite[Theorem 1.2]{BBGHO2020}). \\
On the other hand, results in the opposite direction are not yet complete. In \cite[Theorem 3.6]{Dilworth2003b}, it is proved that every semi-greedy Schauder basis for a Banach space with finite cotype is almost greedy. In \cite[Theorem 1.10]{Berna2019}, the cotype condition is removed and it is proved that every semi-greedy Schauder basis is almost greedy with quasi-greedy and superdemocracy constants depending only on the basis constant and the semi-greedy constant, leaving the question of whether the implication from semi-greedy to almost greedy holds for general Markushevich bases (\cite[Question 1]{Berna2019}). Recently, Bern\'{a} extended \cite[Theorem 1.10]{Berna2019} to a certain class of Markushevich bases (known as $\rho$-admisible) \cite[Theorem 5.3]{Berna2019b}. To our knowledge the general case remained open until now. In this section, we complete the proof of the implication from semi-greedy to almost greedy Markushevich bases, and extend the result to WSG($\tau$) Markushevich bases. We also study the (weak) semi-greedy property for general minimal systems, without the Markushevich hypothesis. We begin with an auxiliary lemma. 

\begin{lemma}\label{lemmaseminormalizedsemigreedy294942}
Let $(x_i)_i\subset X$ be a WSG($\tau$) system with constant $K$. Then both $(x_i)_i$ and $(x_i')_i$ are seminormalized. \\
Moreover, $\sup_{i}\|x_i\|\le 2K\tau^{-1}\inf_{j}(1+\|x_j'\|\|x_j\|)\|x_j\|$.
\end{lemma}

\begin{proof}For every $x\in X$, there is a weak thresholding set $\mathcal{W}^{\tau}(x,1)$, so
$(|x_i'(x)|)_{i}$ is bounded. It follows by uniform boundedness that $(x_i')_{i}$ is bounded. Since
$x_i'(x_i)=1$ for all $i\in \N$, $(x_i)_{i}$ is bounded below. \\
Given $i\not=j$, it follows from Definitions~\ref{defweakthreshset485875} and
\ref{definitionweaksemigreedy747473} that the only $1$-weak thresholding set with weakness parameter
$\tau$ for $x_i+2\tau^{-1}x_j$ is
$$
\mathcal{W}^{\tau}(x_i+2\tau^{-1}x_j,1)=\{ j\}.
$$
Let $ax_j$ be a Chebyshev $\tau$-greedy approximant for $x_i+2\tau^{-1}x_j$. We have
\begin{align}
\|x_i\|\le& \|x_i+2\tau^{-1}x_j-a x_j\|+\|2\tau^{-1}x_j-a x_j\|\nonumber\\
=&\|x_i+2\tau^{-1}x_j-a x_j\|+\|x_j'(x_i+2\tau^{-1}x_j-a x_j)x_j\|\nonumber\\
\le&(1+\|x_j'\|\|x_j\|)\|x_i+2\tau^{-1}x_j-a x_j\|\le (1+\|x_j'\|\|x_j\|)K\sigma_{1}(x_i+2\tau^{-1}x_j)\nonumber\\
\le& 2(1+\|x_j'\|\|x_j\|)\tau^{-1}K\|x_j\|.\nonumber
\end{align}
Thus, $(x_i)_{i}$ is bounded, which implies that $(x_i')_{i}$ is bounded below. Since the above inequality holds also for $i=j$, the bound in the statement follows by taking infimum over $j$ and supremum over $i$. 
\end{proof}

Now we prove that WSG($\tau$) Markushevich bases are almost greedy. The proof combines arguments from the proofs of \cite[Proposition 3.3]{Dilworth2003b} and \cite[Theorem 1.10 b]{Berna2019} - which we adapt to weak thresholding and weak Chebyshev greedy algorithms - with arguments based on the finite dimensional separation property - which allows us to work in the context of general Markushevich bases. 

\begin{theorem}\label{theoremsemigreedyalmostgreedy327211partI}
Let $0<\tau\le 1$, and let $(x_i)_i\subseteq X$ be a WSG($\tau$) system with constant $K$. The following are equivalent:  
\begin{enumerate}[\upshape (i)]
\item \label{almostgreedy} $(x_i)_i$ is almost greedy. 
\item \label{quasigreedy} $(x_i)_i$ is quasi-greedy. 
\item \label{markushevich2} $(x_i)_i$ is a Markushevich basis. 
\item \label{FDSP2} $(x_i)_i$ has the finite dimensional separation property.
\end{enumerate}
If any (and thus all) of these conditions holds, $(x_i)_i$ has second quasi-greedy constant $K_{2q}\le M_{fs}K+M_{fs}(M_{fs}+1)K^{2}\tau^{-2}$ and hyperdemocracy constant $K_{hd}\le M_{fs}(M_{fs}+1)K^2\tau^{-2}$, where $M_{fs}$ is the finite dimensional separation constant of $(x_i)_i$. 
\end{theorem}

\begin{proof}
The implication \ref{almostgreedy} $\Longrightarrow$ \ref{quasigreedy} is immediate. The comments after Theorem~\ref{Theoremalmostgreedydemocraticquasigreedy} give that 
\ref{quasigreedy} $\Longrightarrow$ \ref{markushevich2} and  Proposition~\ref{propositionseparation} gives that
\ref{markushevich2} $\Longrightarrow$ \ref{FDSP2}. To show that
\ref{FDSP2} $\Longrightarrow$ \ref{almostgreedy} fix $0<\epsilon<1$ and let $(x_{i_k})_k$ be a finite dimensional separation subsequence for $(X,M_{fs}, \epsilon)$.\\
Fix $x\in X$ and $m\in\N$, assuming that $\mathcal{G}_{m}(x)\not=0$
(otherwise, $\|x-\mathcal{G}_{m}(x)\|=\|x\|$ and there is nothing to prove). Let
$$
n:=\max{\{1\le j\le m\colon x_{\rho(x,j)}'(x)\not=0\}},
$$
and note that $\mathcal{G}_{n}(x)=\mathcal{G}_{m}(x)$. Since $x_{\rho(x,n)}'(x)\not=0$
and $x_{i}'(x)\xrightarrow[i\to \infty]{} 0$, there is $j_0\in \N$ such that for all $i\ge j_0$, 
\begin{equation}
|x_i'(x)| \le \frac{\tau\epsilon}{2} |x_{\rho(x,n)}'(x)|.\label{muchsmaller327211new}
\end{equation}
Now take $F_0:=[x, x_i\colon 1\le i\le j_0]$, let $\W:=\{ i_{k}\colon j_{F_0}+1\le k\le j_{F_0}+n\}$ and set $z$ as follows.
\begin{equation}
z:= x-\mathcal{G}_{n}(x)+(1+\epsilon)\tau^{-1}|x_{\rho(x,n)}'(x)|
\sum_{j\in\W}x_{j}-\sum_{j\in \W} x_{j}'(x)x_{j}.\label{sufficientlyfar327211}
\end{equation}
Since $i_{j_{F_0}+1}>j_0$, we deduce from \eqref{muchsmaller327211new} and the choice of  $z$ that for every $j\in  \W$ and every  $l\in \N\setminus \W$,
\begin{align*}
\tau |x_j'(z)| & \ge (1+\epsilon)|x_{\rho(x,n)}'(x)|> |x_l'(x-\mathcal{G}_{n}(x))|=|x_l'(z)|.\nonumber
\end{align*}
It follows from this and Definition~\ref{defweakthreshset485875} that the only $n$-weak thresholding set for $z$ with weakness parameter $\tau$ is $\W$. Let $u\in [x_j: j\in \W]$ be an $n$-term Chebyshev $\tau$-greedy approximant for $z$. Notice that both $x$ and $\mathcal{G}_{n}(x)$ belong to $F_0$. Also, by Lemma~\ref{lemmaseminormalizedsemigreedy294942}, $(\|x_i\|)_i$ and $(\|x_i'\|)_i$ are bounded, say by $N$. As
$$
\|\sum_{j\in \W}x_{j}'(x)x_{j}\|\le \sum_{j\in \W}\epsilon|x_{\rho(x,n)}'(x)|\|x_{j}\|\le
\epsilon nN^2\|x\|,
$$
from \eqref{muchsmaller327211new}, \eqref{sufficientlyfar327211} and the choice of our subsequence we deduce that

\begin{align}
\|x-\mathcal{G}_{n}(x)\| &\le(M_{fs}+\epsilon)\|z-u\|\le (M_{fs}+\epsilon)K\sigma_{n}(z)\nonumber\\
&\le (M_{fs}+\epsilon)K\|x+(1+\epsilon)\tau^{-1}|x_{\rho(x,n)}'(x)|\sum_{j\in \W} x_{j}-\sum_{j\in
\W}x_{j}'(x)x_{j}\|\nonumber\\
&\le  (M_{fs}+\epsilon)K(1+\epsilon n N^2)\|x\|+(M_{fs}+\epsilon)K(1+\epsilon)\tau^{-1}|x_{\rho(x,n)}'(x)|\|\sum_{j\in
\W} x_{j}\|.\label{secondleftbound327211}
\end{align}
Now, in order to estimate $|x_{\rho(x,n)}'(x)|\|\sum_{j\in \W}x_{j}\|$ we set
\begin{equation*}
w:= (1-\epsilon)\tau|x_{\rho(x,n)}'(x)|\sum_{j\in \W}x_{j}, 
\end{equation*}
and let $v$ be an $n$-term Chebyshev $\tau$-greedy approximant for $x+w$. 
We claim that $v\in F_0$. To prove this, from~\eqref{muchsmaller327211new} and the definition of $F_0$ we deduce that 
for all $i\in \N_{> j_0}\setminus \W$, 
$$
|x_i'(x+w)|=|x_i'(x)|\le \frac{\tau}{2}|x_{\rho(x,n)}'(x)| < \tau
|x_{\rho(x,n)}'(x)|,
$$
whereas for $i\in \W$,
$$
|x_i'(x+w)|\le (1-\epsilon)\tau
|x_{\rho(x,n)}'(x)|+\frac{\tau\epsilon}{2}|x_{\rho(x,n)}'(x)|<\tau
|x_{\rho(x,n)}'(x)|.
$$
Combining both inequalities above we get that
\begin{equation}
\{ i\in\N\colon |x_i'(x+w)|\ge \tau |x_{\rho(x,n)}'(x)|\} \subseteq \{
1,\dots,j_0\}.\label{beforej0327211}
\end{equation}
On the other hand, for all $1\le i\le n$, we have $i_{j_{F_0}+1}>j_0>\rho(x,i)$, so
$$
|x_{\rho(x,i)}'(x+w)|=|x_{\rho(x,i)}'(x)|\ge |x_{\rho(x,n)}'(x)|.
$$
From this and Definitions~\ref{defweakthreshset485875} and ~\ref{definitionweaksemigreedy747473} we
deduce that
$$
|x_i'(x+w)|\ge \tau |x_{\rho(x,n)}'(x)|
$$
for all $i\in \supp{(v)}$ which, combined with \eqref{beforej0327211}, implies that $\supp{(v)}
\subseteq \{ 1,\dots,j_0\}$. Since, by Definition~\ref{definitionweaksemigreedy747473}, $v$ is a
linear combination of the $x_i$'s with $i$ in its support, it follows that $v\in F_0$ (and so
$x-v\in F_0$). Hence, applying the separation property of $(x_{i_k})_k$ we deduce that
\begin{align*}
(1-\epsilon)\tau|x_{\rho(x,n)}'(x)|\|\sum\limits_{j\in \W}x_{j}\| & =\|w\|\le
\|x+w-v\|+\|x-v\|\\
& \le \|x+w-v\|+(M_{fs}+\epsilon)\|x+w-v\|\\
& \le (1+M_{fs}+\epsilon)K\sigma_{n}(x+w)\\
& \le (1+M_{fs}+\epsilon)K\|x\|.
\end{align*}
Then
$$
|x_{\rho(x,n)}'(x)|\|\sum\limits_{j\in \W}x_{j}\|\le (1-\epsilon)^{-1}\tau^{-1}(1+M_{fs}+\epsilon)K\|x\|.
$$
This result and \eqref{secondleftbound327211} entail that
$$
\|x- \mathcal{G}_{n}(x)\|\le (M_{fs}+\epsilon)K(1+\epsilon n N^2)\|x\|+(M_{fs}+\epsilon)(1+M_{fs}+\epsilon)\frac{1+\epsilon}{1-\epsilon}\tau^{-2}K^2\|x\|.
$$
As $\mathcal{G}_{n}(x)=\mathcal{G}_{m}(x)$, letting  $\epsilon \rightarrow 0$, we get
$$
\|x-\mathcal{G}_{m}(x)\|=\|x-\mathcal{G}_{n}(x)\|\le (M_{fs}K+M_{fs}(M_{fs}+1)K^{2}\tau^{-2})\|x\|.
$$
Since $x$ and $m$ were chosen arbitrarily, this proves that $(x_i)_i$ is quasi-greedy with second quasi-greedy constant $K_{2q}\le M_{fs}K+M_{fs}(M_{fs}+1)K^{2}\tau^{-2}$. Thus, it is a Markushevich basis. \\
Let us show the hyperdemocracy condition. Choose $\epsilon>0$ and $(x_{i_k})_k$ as before, and take $A$ and $B$ finite subsets of $\N$ such that $|A|\le
|B|$, and $(a_i)_{i\in A}, (b_j)_{j\in B}$ scalars with $|a_i|\le |b_j|$ for all $i\in A, j\in B$. Set
$$
F_1:=[x_i\colon i\in A\cup B]\quad \text{and}\quad  a_0:=\max_{i\in A}|a_i|,
$$
assuming without loss of generality that $a_0\not=0$. Take $\W:=\{ i_k\colon j_{F_1}+1\le k\le j_{F_1}+|A| \}$, and define:
$$
z_1:=\tau^{-1}a_0\sum_{l\in \W,} x_{l},\qquad
z_2:= \sum\limits_{i\in A}a_ix_i\qquad \text{and}\qquad
z_3:=\sum\limits_{j\in B}b_jx_j.
$$
Note that for all $i\in A$ and all $l\in \W$,
$$
|x_i'(z_2+(1+\epsilon)z_1)|=|x_i'(z_2)|=|a_i|<a_0(1+\epsilon)=\tau |x_l'(z_2+(1+\epsilon)z_1)|.
$$
By Definition~\ref{defweakthreshset485875}, it follows that the only $|A|$-weak thresholding set for
$z_2+(1+\epsilon)z_1$ with parameter $\tau$ is
$\W$.
Let $u\in [x_i:i\in \W]$ be a $|A|$-term Chebyshev $\tau$-greedy approximant for $z_2+(1+\epsilon)z_1$. We have
\begin{align}
\|\sum\limits_{i\in
A}a_ix_i\|=&\|z_2\|\le(M_{fs}+\epsilon)\|z_2+(1+\epsilon)z_1-u\|\nonumber\\
 \le& (M_{fs}+\epsilon)K\sigma_{|A|}(z_2+(1+\epsilon)z_1)\label{anotherleftside327211}\\
\le& (1+\epsilon)(M_{fs}+\epsilon)K\|z_1\|.\nonumber
\end{align}
Similarly, since for all $j\in B$, $(1-\epsilon)\tau a_0< \tau |b_j|$, the only $|B|$-weak thresholding set
for $z_3+(1-\epsilon)\tau^{2}z_1$ with parameter $\tau$ is $B$. Thus, by the WSG($\tau$) condition there is $v\in [x_i:i\in B]$ such that
\begin{equation*}
\|z_3+(1-\epsilon)\tau^{2}z_1-v\|\le
K\sigma_{|B|}(z_3+(1-\epsilon)\tau^{2}z_1)\le K \|z_3\| = K\|\sum\limits_{j\in B}b_jx_j\|.
\end{equation*}
Hence,
\begin{align}
(1-\epsilon)\tau^{2}\|z_1\|\le& \|z_3+(1-\epsilon)\tau^{2}z_1-v\|+\|z_3-v\|\nonumber\\
\le& (1+M_{fs}+\epsilon)\|z_3+(1-\epsilon)\tau^{2}z_1-v\|\nonumber\\
\le&(1+M_{fs}+\epsilon)K\|\sum\limits_{j\in B}b_jx_j\|.\nonumber
\end{align}
From this and \eqref{anotherleftside327211} we obtain
$$
\|\sum\limits_{i\in A}a_ix_i\|\le
\frac{(1+\epsilon)(M_{fs}+\epsilon)(M_{fs}+1+\epsilon)}{(1-\epsilon)}K^2\tau^{-2}\|\sum\limits_{j\in B}b_jx_j\|.
$$
We complete the proof of the hyperdemocracy property by letting $\epsilon\rightarrow 0$. Finally, an application of Theorem~\ref{Theoremalmostgreedydemocraticquasigreedy} 
gives that $(x_i)_i$ is almost greedy.
\end{proof}
\begin{corollary}\label{corollaryequivalences} Let $(x_i)_i$ be a Markushevich basis. The following are equivalent: 
\begin{enumerate}[\upshape (i)]
\item \label{tausemigreedy487575} For every $0<\tau\le 1$, $(x_i)_i$ is WSG($\tau$).
\item \label{tausemigreedyb487575} There is $0<\tau\le 1$ such that $(x_i)_i$ is WSG($\tau$).
\item \label{weakalmostgreedyalltau487575} For every $0<\tau\le 1$, $(x_i)_i$ is WAG($\tau$).
\item \label{weakalmostgreedyonetau487575} There is $0<\tau\le 1$ such that $(x_i)_i$ is WAG($\tau$).
\item \label{semigreedy487575} $(x_i)_i$ is semi-greedy.
\item \label{almostgreedy487575} $(x_i)_i$ is almost greedy.
\end{enumerate}
\end{corollary}

\begin{proof} 
The implications \ref{tausemigreedy487575} $\Longrightarrow$ \ref{tausemigreedyb487575} and \ref{weakalmostgreedyalltau487575} $\Longrightarrow$ \ref{weakalmostgreedyonetau487575} are immediate. \\ Also,  \ref{semigreedy487575} $\Longrightarrow$ \ref{tausemigreedy487575} and  \ref{almostgreedy487575} $\Longrightarrow$ \ref{weakalmostgreedyalltau487575} follow at once from the definitions.\\
That \ref{tausemigreedyb487575} $\Longrightarrow$ \ref{almostgreedy487575}  and \ref{semigreedy487575} $\Longrightarrow$ \ref{almostgreedy487575} follow by Theorem~\ref{theoremsemigreedyalmostgreedy327211partI}. \\
By Proposition~\ref{propositionwagag394995}, we see that       \ref{weakalmostgreedyonetau487575} $\Longrightarrow$ \ref{almostgreedy487575}. \\
Finally, \ref{almostgreedy487575} $\Longrightarrow$ \ref{semigreedy487575} follows by  \cite[Theorem 3.2]{Dilworth2003b}, as their proof holds for Markushevich bases. 
\end{proof}

 When the conditions of Proposition~\ref{propositionseparation}\ref{weakzero} hold, we have $M_{fs}=1$ for any Markushevich basis. Thus, in such cases Theorem~\ref{theoremsemigreedyalmostgreedy327211partI} gives upper bounds for the second quasi-greedy and the hyperdemocracy constant depending only on $K$ and $\tau$. On the other hand, and unlike the implication from (weak) almost greedy to semi-greedy, in general there is no upper bound for the almost greedy constant of a WSG($\tau$) Markushevich basis depending only on the WSG($\tau)$ constant and $\tau$. The following example illustrates that. 

\begin{example}\label{exampleseminotalmost493821} Let $(e_i)_{i}$ and $(e_i')_{i}$ be the unit vector
basis of $\ell_1$ and its sequence of coordinate functionals respectively. Given $\alpha>0$, define
\begin{align*}
x_{i}:=&e_i+ 2(\alpha+1)(-1)^{i}e_1\; \text{ for \ all\ }  i\ge 2;\\
X:=& \overline{[x_{i}\colon i\ge 2]};\\
x_{i}':=&e_i'\big|_{X}\; \text{ for \ all\ }  i\ge 2.
\end{align*}
The following statements hold:
\begin{enumerate}[\upshape (a)]
\item \label{l1493821} $(x_i)_{i\ge 2}$ is a basic sequence equivalent to the unit vector basis of $\ell_1$ with democracy constant $K_{d}> \alpha+1$.
\item \label{quasigreedy493821}  $(x_i)_{i\ge 2}$ is an almost greedy Markushevich basis for $X$ with biorthogonal functionals  $(x_i')_{i\ge 2}$. Its almost greedy constant and quasi-greedy constants are greater than $\alpha$.\item \label{WAG493821} For every $0<\tau<1$, if $M(\tau)$ is a WAG($\tau$) constant for
    $(x_i)_{i\ge 2}$, then
$$
M(\tau) > \tau \sqrt{\alpha+1}.
$$
\item \label{semigreedy493821} $(x_i)_{i\ge 2}$ is a semi-greedy Markushevich basis for $X$ with
    semi-greedy constant $K_{s}\le 4$. Moreover, for every $x\in X$, $m\in \N$ and every set
    $\mathcal{W}^{\tau}(x,m)$, there are scalars $(b_i)_{i\in \mathcal{W}^{\tau}(x,m)}$ such that
$$
\|x-\sum\limits_{i\in \mathcal{W}^{\tau}(x,m)}b_ix_i\|\le 4\tau^{-1}\sigma_{m}(x).
$$
\end{enumerate}
\end{example}

\begin{proof} Let us show that \ref{l1493821} holds.
For each $n\ge 2$, we have
$$
\sum\limits_{i=2}^{n}|a_i|\le
\sum\limits_{i=2}^{n}|a_i|+|\sum\limits_{i=2}^{n}2(\alpha+1)(-1)^{i}a_i|=\|\sum\limits_{i=2}^{n}a_ix_{i}\|\le
(3+2\alpha)\sum\limits_{i=2}^{n}|a_i|.
$$
We see that  $(x_i)_{i\ge 2}$ is basic and equivalent to the unit vector basis of $\ell_1$, so in
particular, it is democratic. Since
$\|x_2+x_3\|=\|e_2+2(\alpha+1)e_1+e_3-2(\alpha+1)e_1\|=2$ and $\|x_2\|=\|e_2+2(\alpha+1)e_1\|=2\alpha+3$, it follows that
$$
K_{d}\ge \frac{2\alpha+3}{2}>\alpha+1.
$$
As $(x_i)_{i\ge 2}$  is equivalent to $(e_i)_{i\ge 2}$, it is an almost greedy Markushevich basis for $X$. Clearly,  $(x_i')_{i\ge 2}$ is its biorthogonal sequence. To complete the proof of~\ref{quasigreedy493821}, let $K_a, K_{1q}$ and $K_{2q}$ be the almost greedy, first and second quasi-greedy constants of the system, respectively. By Theorem~\ref{Theoremalmostgreedydemocraticquasigreedy}, we have  $K_a\ge K_d>\alpha+1$, and since $\mathcal{G}_1(x_2+x_3)=x_2,$ we get that $K_{1q}\ge \frac{2\alpha+3}{2}>\alpha+1$, so $K_{2q}>\alpha$. \\
To prove \ref{WAG493821}, apply Proposition~\ref{propositionwagag394995} to get 
$$
\alpha+1 < K_d \le K_{hd}\le \tau^{-2}M^2(\tau), 
$$
from where the lower bound for $M(\tau)$ is obtained.\\
Finally, let us show that \ref{semigreedy493821} holds.
Fix $0<\tau\le 1$, $x\in X$, $m\in \N$, and a set $\W=\mathcal{W}^{\tau}(x,m)$. Given
$A\subseteq \N_{>1}$  with $|A|=m$ and scalars $(a_i)_{i\in A}$, we proceed as follows:
If $A=\mathcal{W}$, we choose $b_i:=a_i$ for each $i$. Otherwise, fix $\pi$ any bijection
$$
\pi\colon \W\setminus A\rightarrow A\setminus \W.
$$
For every $j\in \W$, we define
\begin{numcases}{b_j:=}
a_j & if $j\in A$;\nonumber\\
(-1)^{j+\pi(j)}a_{\pi(j)} & otherwise.\nonumber
\end{numcases}
Let us estimate the $\ell_1$-norm  $\|x-\sum\limits_{j\in\W} b_jx_{j}\|$ in terms of $\|x-\sum\limits_{i\in A}a_ix_{i}\|$. For the first coordinate we get
\begin{align}
e_1'(x-\sum\limits_{j\in \W}b_jx_{j})
= & e_1'(x)-\sum\limits_{i\in \W \cap A}a_ie_1'(x_{i})-\sum\limits_{j\in \W\setminus
A}a_{\pi(j)}(-1)^{j+\pi(j)}e_1'(x_{j})\nonumber\\
= & e_1'(x)-\sum\limits_{i\in \W\cap A}a_ie_1'(x_{i})-\sum\limits_{j\in \W\setminus
A}2a_{\pi(j)}(\alpha+1)(-1)^{j+\pi(j)}(-1)^{j}\nonumber\\
= & e_1'(x)-\sum\limits_{i\in \W\cap A}a_ie_1'(x_{i})-\sum\limits_{i\in  A\setminus\W}
2a_i(\alpha+1)(-1)^{i}\nonumber\\
= & e_1'(x-\sum\limits_{i\in A}a_ix_{i}). \nonumber
\end{align}
Now, suppose that  $l>1$. For $l\in \N\setminus A\cup \W$, 
\begin{equation*}
e_l'(x-\sum\limits_{j\in \W} b_jx_{j}) = e_l'(x) = e_l'(x-\sum\limits_{i\in A} a_ix_{i}).
\end{equation*}
Also, if  $l\in A\cap \W$ then
\begin{equation*}
e_l'(x-\sum\limits_{j\in \W}b_jx_{j})=e_l'(x)-b_l=e_l'(x)-a_l=e_l'(x-\sum\limits_{i\in
A}a_ix_{i}).
\end{equation*}
On the other hand, if $l\in \W\setminus A$ we compute the $l$- and the $\pi(l)$- coordinates
components at the same time. From the fact that $\pi(l)\not \in \W$ we deduce that
$|e_l'(x)|=|x_l'(x)|\ge \tau|x_{\pi(l)}'(x)|=\tau|e_{\pi(l)}'(x)|$. Thus,
\begin{align}
|e_l'(x-\sum\limits_{j\in \W} b_jx_{j})|+|e_{\pi(l)}'(x-\sum\limits_{j\in \W} b_jx_{j})| = &
|e_l'(x)-b_l|+|e_{\pi(l)}'(x)|\nonumber\\
= & |e_l'(x)-(-1)^{l+\pi(l)}a_{\pi(l)}|+|e_{\pi(l)}'(x)|\nonumber\\
\le & 2 \frac{|e_l'(x)|}{\tau}+|a_{\pi(l)}|\nonumber\\
\le& 2 \max\Big\{2 \frac{|e_l'(x)|}{\tau},|a_{\pi(l)}|\Big\}\nonumber\\
\le & 2 \max\Big\{2 \frac{|e_l'(x)|}{\tau},2|a_{\pi(l)}|-2\frac{|e_l'(x)|}{\tau}\Big\}\nonumber\\
= & \frac{4}{\tau}\tau \max\Big\{\frac{|e_l'(x)|}{\tau},|a_{\pi(l)}|-\frac{|e_l'(x)|}{\tau}\Big\}.
\label{oneside493821}
\end{align}
Similarly,
\begin{align}
|e_l'(x-\sum\limits_{i\in A}a_ix_{i})|+|e_{\pi(l)}'(x-\sum\limits_{i\in
A}a_ix_{i})|=&|e_l'(x)|+|e_{\pi(l)}'(x)-a_{\pi(l)}|\nonumber\\
\ge& \max\Big\{|e_l'(x)|,|a_{\pi(l)}|-|e_{\pi(l)}'(x)|\Big\}\nonumber\\
\ge& \max\Big\{|e_l'(x)|,|a_{\pi(l)}|-\frac{|e_{l}'(x)|}{\tau}\Big\}\nonumber\\
\ge& \tau\max\Big\{\frac{|e_l'(x)|}{\tau},|a_{\pi(l)}|-\frac{|e_{l}'(x)|}{\tau}\Big\}.
\label{anotherside493821}
\end{align}
Comparing \eqref{oneside493821} and \eqref{anotherside493821}, we obtain
$$
|e_l'(x-\sum\limits_{j\in \W} b_jx_{j})|+|e_{\pi(l)}'(x-\sum\limits_{j\in \W} b_jx_{j})|\le
\frac{4}{\tau}(|e_l'(x -\sum\limits_{i\in A}a_ix_{i})|+|e_{\pi(l)}'(x -\sum\limits_{i\in
A}a_ix_{i})|).
$$
Combining the above estimates we get
$$
\|x-\sum\limits_{j\in\W} b_jx_{j}\|\le \frac{4}{\tau}\|x-\sum\limits_{i\in A}a_ix_{i}\|.
$$
Now, the  left-hand side of the inequality is greater than or equal to the infimum over all scalars
$(b_i)_{i\in \W}$, which in fact is a minimum since $\W$ is finite. Then, taking the infimum over all
$A\subseteq \N$ with $|A|=m$ and scalars $(a_i)_{i\in A}$ on the right-hand side, we conclude that$$
\min\limits_{(b_j)_{j\in \W}\subseteq\mathbb{K}}\|x-\sum\limits_{j\in  \W}b_jx_j\|\le
\frac{4}{\tau}\sigma_{m}(x).
$$
Taking $\tau=1$, we conclude that $(x_i)_{i\ge 2}$ is semi-greedy, and we get the bound for $K_s$.
\end{proof}
A natural question in this context is whether the implication from WSG($\tau$) to almost greedy holds for all WSG($\tau$) systems, or - equivalently in light of Theorem~\ref{theoremsemigreedyalmostgreedy327211partI} -  whether every weak semi-greedy system is a Markushevich basis. The answer is negative. The following example shows a semi-greedy system that is
neither quasi-greedy nor democratic.
\begin{example}\label{exampleseminotalmost939214} Let $(e_i)_{i}$ be the unit vector basis of
$\mathtt{c}_0$ and let $(e_i')_{i}$ be the sequence of biorthogonal functionals. Set
\begin{align*}
x_i:=&e_i+(-1)^{i}e_1\ \text{ for \ all\ }  i\ge 2;\\
x_i':=& e_i'\ \text{ for \ all\ }  i\ge 2.
\end{align*}
The following statements hold:
\begin{enumerate}[\upshape (a)]
\item \label{markushevich939214} $(x_i)_{i\ge 2}$ is a fundamental minimal system for
    $\mathtt{c}_0$, but not a Markushevich basis. Thus, it is not quasi-greedy.
\item \label{democratic939214} $(x_i)_{i\ge 2}$  is not democratic.
\item \label{semigreedy939214} $(x_i)_{i\ge 2}$  is a semi-greedy system for $\mathtt{c}_0$ with
    semi-greedy constant no greater than $3$. Moreover, for any $x\in X$, $m\in \N$ and every set
    $\mathcal{W}^{\tau}(x,m)$, there are scalars $(b_i)_{i\in \mathcal{W}^{\tau}(x,m)}$ such that
$$
\|x-\sum\limits_{i\in \mathcal{W}^{\tau}(x,m)}b_ix_i\|\le 3\tau^{-1}\sigma_{m}(x).
$$
\end{enumerate}
\end{example}

\begin{proof}
To show that \ref{markushevich939214} holds, first note that for all $n\in \N$, 
$$
\|e_1-\sum\limits_{i=1}^{n}\frac{x_{2i}}{n}\|=\|\sum\limits_{i=1}^{n}\frac{e_{2i}}{n}\|=\frac{1}{n}.
$$
This entails that $e_1\in \overline{[x_i\colon i\ge 2]}$, so $(x_i)_{i\ge 2}$  is fundamental. Since
$x_j'(e_1)=0$ for every $j\ge 2$, $(x_i)_{i\ge 2}$  is not a Markushevich basis, and thus it is not quasi-greedy.\\
To see that $(x_i)_{i\ge 2}$ is not democratic, notice that for all $n\in \N$, 
$$
\|\sum\limits_{i=2}^{2n+1}x_{i}\|=\|\sum\limits_{i=2}^{2n+1}e_{i}\|=1,
$$
but
$$
\|\sum\limits_{i=1}^{2n}x_{2i}\|=\|2ne_1+\sum\limits_{i=1}^{2n}e_{i}\|=2n.
$$
Hence, \ref{democratic939214} holds.
To prove \ref{semigreedy939214}, we proceed as in the proof of Example
\ref{exampleseminotalmost493821}. Fix $0<\tau\le 1$, $x\in X$, $m\in \N$ and a set
$\W=\mathcal{W}^{\tau}(x,m)$. Take a set $A\subseteq \N_{>1}$ with $|A|=m$ and $A\not=\W$, and
scalars $(a_i)_{i\in A}$, and let
$$
\pi\colon\W\setminus A\rightarrow A\setminus \W
$$
be a bijection. For every $j\in \W$, define
\begin{numcases}{b_j:=}
a_j & if $j\in A$;\nonumber\\
(-1)^{j+\pi(j)}a_{\pi(j)} & otherwise.\nonumber
\end{numcases}
Now, we estimate the supremum norm of $x-\sum\limits_{j\in \W}b_jx_{j}$ in terms of that of
$x-\sum\limits_{i\in A}a_jx_{j}$.
First note that if $l>1$ and  $l\in \N\setminus A\cup\W$ or $l\in  A\cap \W$, we have
\begin{equation*}
e_l'(x-\sum\limits_{j\in \W}b_jx_{j})=e_l'(x-\sum\limits_{i\in A}a_ix_{i}).
 \label{notmodified939214}
\end{equation*}
This equality also holds for $l=1$, indeed
\begin{align*}
e_1'(x-\sum\limits_{j\in \W}b_jx_{j})  = &e_1'(x)-\sum\limits_{i\in\W\cap
A}a_ie_1'(x_{i})-\sum\limits_{j\in\W\setminus A}a_{\pi(j)}(-1)^{\pi(j)}\nonumber\\
=&e_1'(x)-\sum\limits_{i\in\W\cap A}a_ie_1'(x_{i})-\sum\limits_{j\in\W\setminus
A}a_{\pi(j)}e_1'(x_{\pi(j)})\nonumber\\
=&e_1'(x)-\sum\limits_{i\in \W\cap A}a_ie_1'(x_{i})-\sum\limits_{i\in
A\setminus\W}a_ie_1'(x_i)\nonumber\\
= &e_1'(x-\sum\limits_{i\in A}a_ix_i).
\end{align*}
For each $l\in\W\setminus A$, we have $|e_l'(x)|=|x_l'(x)|\ge \tau|x_{\pi(l)}'(x)|=\tau
|e_{\pi(l)}'(x)|$. Hence, considering together the $l$- and the $\pi(l)$-th coordinates we have
\begin{align*}
\max\Big\{|e_l'(x-\sum\limits_{j\in \W}b_jx_{j})|,|e_{\pi(l)}'(x-\sum\limits_{j\in
\W}b_jx_{j})|  \Big\} &=\max\Big\{|e_l'(x)-b_l|,|e_{\pi(l)}'(x)|\Big\}\\
&\le \max\Big\{|e_l'(x)-\pm a_{\pi(l)}|,\frac{|e_{l}'(x)|}{\tau}\Big\}\nonumber\\
&  \le   \frac{|e_l'(x)|}{\tau}+|a_{\pi(l)}|\\
& \le 3
\max\Big\{\frac{|e_{l}'(x)|}{\tau},|a_{\pi(l)}|-\frac{|e_l'(x)|}{\tau}\Big\}.
\end{align*}
Similarly, we obtain
$$
\max\Big\{|e_l'(x-\sum\limits_{i\in A}a_ix_{i})|,|e_{\pi(l)}'(x-\sum\limits_{i\in
A}a_ix_{i})|\Big\}\ge \tau
\max\Big\{\frac{|e_l'(x)|}{\tau},|a_{\pi(l)}|-\frac{|e_{l}'(x)|}{\tau}\Big\}.
$$
From the inequalities given above,
$$
\|x-\sum\limits_{j\in \W}b_jx_{j}\|\le \frac{3}{\tau}\|x-\sum\limits_{i\in A}a_ix_{i}\|.
$$
The proof of \ref{semigreedy939214} is completed  by the same argument given in Example~
\ref{exampleseminotalmost493821}.
\end{proof}
\begin{remark}The system of Example~\ref{exampleseminotalmost939214} can also be considered in
$\ell_p$ for all $1<p<\infty$. With only minor adjustments to the calculations given above, we obtain
that  $(x_i)_{i\ge 2}\subseteq \ell_p$  is semi-greedy (with constant no greater than $3*2^{\frac{1}{p}}$),
but neither democratic nor quasi-greedy.
\end{remark}
Our next proposition shows that from any WSG($\tau$) system that is not a Markushevich basis, one can obtain an almost greedy Markushevich basis for the space, with superdemocracy and first quasi-greedy constants depending only on $\tau$ and the WSG($\tau$) constant of the system - and thus, by Theorem~\ref{Theoremalmostgreedydemocraticquasigreedy}, with almost greedy constant also depending only on said constants. In order to prove our result, we need two technical lemmas. The notation used below is natural and according to the context.
\begin{lemma} \label{lemmastillalmostgreedy947756} Let $\mathcal{B}_1=(x_i)_{i\in \N}$ be a
fundamental minimal system for $Y$, and suppose that
$\mathcal{B}_2:=(x_0,x_i)_{i\in \N}$ is a fundamental minimal system for $X:=\overline{[x_i\colon i\in
\N_0]}$ with biorthogonal functionals $(x_0',x_i')_{i\in \N}\subseteq X'$ satisfying
$$\|x_0\|\|x_0'\|=1\qquad \text{and}\qquad \|x_0\|=\sup\limits_{i\in \N}\|x_i\|.
$$
The following hold:
\begin{enumerate}[\upshape (a)]
\item \label{1947756}
If $\mathcal{B}_1$ is quasi-greedy with first quasi-greedy constant $K_{1q}(\mathcal{B}_1)$, then
$\mathcal{B}_2$ is quasi-greedy with first quasi-greedy constant
$$
K_{1q}(\mathcal{B}_2)\le 2K_{1q}(\mathcal{B}_1)+1.
$$
\item \label{2947756} If $\mathcal{B}_1$ is superdemocratic with constant $K_{sd}(\mathcal{B}_1)$,
    then $\mathcal{B}_2$ is superdemocratic with constant
\begin{equation*}
K_{sd}(\mathcal{B}_2)\le 4K_{sd}(\mathcal{B}_1).
\end{equation*}
\end{enumerate}
\end{lemma}
\begin{proof}
 To prove \ref{1947756}, fix  $x\in X$ and $m\in \N$. Then, we have
\begin{numcases}{\mathcal{G}_{\mathcal{B}_2,m}(x)=}
\mathcal{G}_{\mathcal{B}_1,m-1}(x-x_0'(x)x_0)+x_0'(x)x_0 & if $0\in
\mathcal{GS}_{\mathcal{B}_2,m}(x)$;\nonumber\\
\mathcal{G}_{\mathcal{B}_1,m}(x-x_0'(x)x_0) & otherwise.\nonumber
\end{numcases}
Thus,
$$
\|\mathcal{G}_{\mathcal{B}_2,m}(x)\|\le K_{1q}(\mathcal{B}_1)\|x-x_0'(x)x_0\|+\|x_0'(x)x_0\| \le
(2K_{1q}(\mathcal{B}_1)+1)\|x\|.
$$
It follows that $\mathcal{B}_2$ is quasi-greedy and $K_{1q}(\mathcal{B}_2)\le
2K_{1q}(\mathcal{B}_1)+1$.

To prove \ref{2947756}, suppose first that $D\subseteq \N_0$ is a finite nonempty set and take scalars
$(a_k)_{k\in D}$  with $|a_k|=1$ for each $k\in D$. If $0\in D$, then
\begin{equation}
\|x_0\|=\|x_0\||x_0'(\sum\limits_{k\in D}a_kx_k)|\le \|\sum\limits_{k\in D}a_kx_k\|.
\label{firstsd947756}
\end{equation}
On the other hand, if $0\not\in D$ we have
\begin{equation}
\|x_0\|=\sup\limits_{k\in\N}\|x_k\|\le K_{sd}(\mathcal{B}_1)\|\sum\limits_{k\in
D}a_kx_k\|.\label{secondsd947756}
\end{equation}
Now let $A, B\subseteq \N_{0}$ be finite nonempty sets with $|A|\le |B|$, and take $(a_i)_{i\in A}$, $(b_j)_{j\in B}$ scalars such that $|a_i|=|b_j|=1$ for all $i\in A, j\in B$. If $0\not\in A\cup B$,
there is nothing to prove. If $0\in A\setminus B$, by \eqref{secondsd947756} we have
\begin{align*}
\|\sum\limits_{i\in A}a_ix_i\|\le& \|x_0\|+\|\sum\limits_{i\in A\setminus \{0\}}a_ix_i\|\le
\|x_0\|+K_{sd}(\mathcal{B}_1)\|\sum\limits_{j\in B}b_jx_j\|\\
\le& 2K_{sd}(\mathcal{B}_1)\|\sum\limits_{j\in B}b_jx_j\|.
\end{align*}
If $0\in A\cap B$, by \eqref{firstsd947756} and \eqref{secondsd947756} we get
\begin{align*}
\|\sum\limits_{i\in A}a_ix_i\|\le& \|x_0\|+\|\sum\limits_{i\in A\setminus\{0\}}a_ix_i\|\nonumber\\
\le& \|\sum\limits_{j\in B}b_jx_j\|+K_{sd}(\mathcal{B}_1)\|\sum\limits_{j\in B\setminus
\{0\}}b_jx_j\|\nonumber\\
\le& (1+K_{sd}(\mathcal{B}_1))\|\sum\limits_{j\in B}b_jx_j\|+K_{sd}(\mathcal{B}_1)\|x_0\|\nonumber\\
\le& (1+2K_{sd}(\mathcal{B}_1))\|\sum\limits_{j\in B}b_jx_j\|.
\end{align*}
If $0\in B\setminus A$ and $|B|>1$, we choose $i_0\in A$ and using \eqref{firstsd947756} we get that
\begin{align*}
\|\sum\limits_{i\in A}a_ix_i\|\le& \|x_{i_0}\|+\|\sum\limits_{i\in A\setminus \{i_0\}}a_ix_i\|\le
2K_{sd}(\mathcal{B}_1)\|\sum\limits_{j\in B\setminus \{0\}}b_jx_j\|\nonumber\\
\le& 2K_{sd}(\mathcal{B}_1)(\|\sum\limits_{j\in B}b_jx_j\|+ \|x_0\|)\nonumber\\
 \le& 4K_{sd}(\mathcal{B}_1)\|\sum\limits_{j\in B}b_jx_j\|.
\end{align*}
The only case left is $A\not=B=\{0\}$.  Then $A=\{i_0\}$ for some $i_0\in \N$ and
\begin{equation*}
\|\sum\limits_{i\in A}a_ix_i\|=\|x_{i_0}\|\le \sup\limits_{k\in\N}\|x_k\|=\|x_0\|=\|\sum\limits_{j\in
B}b_jx_j\|.
\end{equation*}
From the above estimations, $\mathcal B_2$ is superdemocratic and $K_{sd}(\mathcal{B}_2)\le
4K_{sd}(\mathcal{B}_1)$.
\end{proof}
The following result will allow us to handle the case $\sigma_m(x)=0$ in the proof of
Proposition~\ref{propositionsemigreedyalmostgreedy327211partIIIc}.
\begin{lemma}\label{lemmanullapproximant397562}
Let $(x_i)_{i}\subseteq X$ be a fundamental minimal system with both $(x_i)_i$ and $(x_i')_i$ bounded. If $x\in X$  is such that $\sigma_m(x)=0$ for some $m\in \N$, then $|\supp{(x)}|\le m$ and $x=\G_m(x)=P_{\supp{(x)}}(x)$. 
\end{lemma}
\begin{proof}
Let $B:=\supp(x)$. If $|B|>m$, there is $C\subset B$ with $|C|=m+1$. Thus, if $A\subset \N$ and $|A|\le m$, there is $j\in C\setminus A$. Then, for any scalars $(a_i)_{i\in A}$ it follows that 
$$
\|x-\sum\limits_{i\in A}a_ix_i\|\ge \frac{|x_j'(x-\sum\limits_{i\in A}a_ix_i)|}{\|x_j'\|}=\frac{|x_j'(x)|}{\|x_j'\|}\ge \frac{\min_{i\in C}|x_i'(x)|}{\max_{i\in C}\|x_i'\|}>0.
$$
Taking infimum over such sets and scalars, we get a contradiction to the hypothesis that $\sigma_m(x)=0$.  Now let 
$$
M:=\sup_{i}\{\|x_i\|,\|x_i'\|\}.
$$
Given that $\sigma_m(x)=0$ and $|B|\le m$, we have $\sigma_{2m}(x-P_B(x))=0$. Fix $\epsilon>0$ and choose $A\subset X$ with $|A|=2m$ and scalars $(a_i)_{i\in A}$ so that 
$$
\|x-P_B(x)-\sum\limits_{i\in A}a_ix_i\|\le \epsilon.
$$ 
For each $l\in A$, we have
\begin{equation*}
|a_l|=|x_l'(x-P_B(x)-\sum\limits_{i\in A}a_ix_i)|\le M\|x-P_B(x)-\sum\limits_{i\in A}a_ix_i\|\le M\epsilon. 
\end{equation*}
Hence,
$$
\|x-P_B(x)\|\le \epsilon+\|\sum\limits_{i\in A}a_ix_i\|\le \epsilon+\sum\limits_{i\in A}M\epsilon\|x_i\|\le \epsilon+2mM^2\epsilon.
$$
Since $\epsilon$ is arbitrary and $m$, $M$ are fixed, we get $x=P_B(x)$, and thus $x=\G_m(x)$. 
\end{proof}

Now we can show that a weak semi-greedy system that is not a Markushevich basis can be slightly modified to obtain a Markushevich basis (and therefore, by Theorem~\ref{theoremsemigreedyalmostgreedy327211partI}, an  almost greedy  system).

\begin{proposition}\label{propositionsemigreedyalmostgreedy327211partIIIc}  Let $0<\tau\le 1$, and let $\mathcal{B}:=(x_i)_{i}\subseteq X$ be a WSG($\tau$) system that is not a Markushevich basis, with constant $K_{ws}(\tau,\mathcal{B})$. There are $x_0\in X$ and $x_0'\in X'$ such that
\begin{equation}
\overline{\{x_i\}}^{w}_{i\in\N}\subseteq \{x_i\}_{i\in\N}\cup [x_0],\nonumber
\end{equation}
and the system
\begin{equation}
\mathcal{B}_1:=(x_0,x_i-x_0'(x_i)x_0)_{i\in \N}\nonumber
\end{equation}
is an almost greedy Markushevich basis for $X$ with biorthogonal functionals  $(x_0', x_i')_{i}$. In addition, $\mathcal{B}_1$ has first quasi-greedy constant
$$
K_{1q}(\mathcal{B}_1)\le 3+4K_{ws}(\tau,\mathcal{B})+16K_{ws}(\tau,\mathcal{B})^2\tau^{-2},
$$
and superdemocracy constant 
$$
K_{sd}(\mathcal{B}_1)\le 32K_{ws}(\tau,\mathcal{B})^2\tau^{-2}.
$$
\end{proposition}
\begin{proof}
By Theorem~\ref{theoremsemigreedyalmostgreedy327211partI} and Corollary~\ref{corollarywhenFDSP}, the set $\overline{\{x_i\}}^{w}_{i\in\N}$ is weakly compact, and $0\not \in \overline{\{x_i\}}^{w}_{i\in\N}$. Then,
there is a subnet $(x_{i_{\lambda}})_{\lambda}$ and $u_0\in X\setminus \{0\}$ such that
$$
x_{i_{\lambda}}\xrightarrow{w}u_0.
$$
By the Hahn--Banach Theorem, there is $u_0'\in X'$ such that
$$
\|u_0\|\|u_0'\|= 1 \quad \text{and} \quad u_0'(u_0)= 1.
$$
Now  for  $x\in X$ define the linear operator $P\colon X\rightarrow X$ by
$$
P(x):=x-u_0'(x)u_0.
$$
It is easy to check that $P$ is a projection, $\|P\|\le 2$, $X=[u_0]\oplus P(X)$
and, as $x_j'(u_0)=0$ for all $j\in \N$,  
$$
\mathcal{B}_2:=(x_i-u_0'(x_i)u_0)_{i\in \N}
$$
is a fundamental minimal system for $P(X)$ with biorthogonal functionals  $(x_i'\big|_{P(X)})_{i}$.\\
First, we show that $\mathcal{B}_2$ is a WSG($\tau$) system for $P(X)$.
Take $y\in P(X)$ and $m\in \N$, and fix $\epsilon>0$. If $\sigma_{\mathcal{B}_2,m}(y)\not=0$ we choose
a set $A\subseteq \N$ with $|A|=m$ and scalars $(a_i)_{i\in A}$ so that
$$
\|y-\sum\limits_{i\in A}a_i(x_i-u_0'(x_i)u_0)\|\le (1+\epsilon)\sigma_{\mathcal{B}_2,m}(y),
$$
and define
$$
z:=y+\sum\limits_{i\in A}a_iu_0'(x_i)u_0.
$$
By hypothesis, there is a set $\W^{\tau}_{\mathcal{B}}(z,m)$ and scalars $(b_j)_{j\in
\W^{\tau}_{\mathcal{B}}(z,m)}$ such that
$$
\|z-\sum\limits_{j\in \W^{\tau}_{\mathcal{B}}(z,m)}b_jx_j\|\le
K_{ws}(\tau,\mathcal{B})\sigma_{\mathcal{B},m}(z).
$$
Then, it follows that
\begin{align*}
\|y-\sum\limits_{j\in \W^{\tau}_{\mathcal{B}}(z,m)}b_j(x_j-u_0'(x_j)u_0)\|& =\|P(z-\sum\limits_{j\in
\W^{\tau}_{\mathcal{B}}(z,m)}b_jx_j)\|\nonumber\\
&\le 2\|z-\sum\limits_{j\in \W^{\tau}_{\mathcal{B}}(z,m)}b_jx_j\|\nonumber\\
&\le 2K_{ws}(\tau,\mathcal{B})\|y-\sum\limits_{i\in A}a_i(x_i-u_0'(x_i)u_0)\|\nonumber\\
&\le 2(1+\epsilon)K_{ws}(\tau,\mathcal{B})\sigma_{\mathcal{B}_2,m}(y).
\end{align*}
Since $x_i'(y)=x_i'(z)$ for every $i\in\N$, the set $\W^{\tau}_{\mathcal{B}}(z,m)$ is also a weak thresholding set for $y$ with respect to $\mathcal{B}_2$, and we obtain the estimate

\begin{equation}
\|y-\sum\limits_{j\in \W^{\tau}_{\mathcal{B}_2}(y,m)}b_j(x_j-u_0'(x_j)u_0)\|\le
2(1+\epsilon)K_{ws}(\tau,\mathcal{B})\sigma_{\mathcal{B}_2,m}(y).\label{tausemigreedy327211IIIc}
\end{equation}
Now suppose that $\sigma_{\mathcal{B}_2,m}(y)=0$. By Lemma~\ref{lemmaseminormalizedsemigreedy294942}, both $(x_i)_i$ and $(x_i')_i$ are seminormalized. The former implies that $(x_i-u_0'(x_i)u_0)_i$ is bounded, so by Lemma~\ref{lemmanullapproximant397562}, we have
$$
y=\mathcal{G}_{\mathcal{B}_2,m}(y)=\sum\limits_{j\in
\mathcal{GS}_{\mathcal{B}_2,m}(y)}x_j'(y)(x_j-u_0'(x_j)u_0).
$$
Hence, \eqref{tausemigreedy327211IIIc} holds also in this case taking $\W^{\tau}_{\mathcal{B}_2}(y,m):=\mathcal{GS}_{\mathcal{B}_2,m}(y)$ and $b_j:=x_j'(y)$ for all $j\in \W^{\tau}_{\mathcal{B}_2}(y,m)$. Then we conclude that $\mathcal{B}_2$
is a WSG($\tau$) system for $P(X)$ with constant
$$
K_{ws}(\tau,\mathcal{B}_2)\le 2(1+\epsilon)K_{ws}(\tau,\mathcal{B}).
$$
As $x_{i_{\lambda}}\xrightarrow{w}u_0$, we get that $
x_{i_\lambda}-u_0'(x_{i_\lambda})u_0 \xrightarrow{w} 0$.
Then, by Corollary~\ref{corollarywhenFDSP}, Theorem~\ref{theoremsemigreedyalmostgreedy327211partI} and Proposition~\ref{propositionseparation}\ref{weakzero}, it follows that
$\mathcal{B}_2$ is an almost greedy Markushevich basis for $P(X)$, with first quasi-greedy constant
\begin{align}
K_{1q}(\mathcal{B}_2)\le& 1+K_{ws}(\tau,\mathcal{B}_2)+2K_{ws}(\tau,\mathcal{B}_2)^2\tau^{-2}\nonumber\\
\le&
1+2(1+\epsilon)K_{ws}(\tau,\mathcal{B})+8(1+\epsilon)^2K_{ws}(\tau,\mathcal{B})^2\tau^{-2}\label{fqg327211IIIc}
\end{align}
and with superdemocracy constant 
\begin{equation}
K_{sd}(\mathcal{B}_2)\le 2K_{ws}(\tau,\mathcal{B}_2)^2\tau^{-2}\le
8(1+\epsilon)^2K_{ws}(\tau,\mathcal{B})^2\tau^{-2}.
\label{sd327211IIIc}
\end{equation}
Now set
$$
a_0:=\frac{1}{\|u_0\|}\sup\limits_{i\in \N}\|x_i-u_0'(x_i)u_0\|,\qquad x_0:=a_0u_0,\qquad\text{and} \qquad x_0':=a_0^{-1}u_0'.
$$ 
Since 
$$
\|x_0\|\|x_0'\|=1, \quad \|x_0\|=\sup\limits_{i\in\N}\|x_i-x_0'(x_i)x_0\|\quad\text{and}\qquad \mathcal{B}_2=(x_i-x_0'(x_i)x_0)_{i},
$$
we may apply Lemma~\ref{lemmastillalmostgreedy947756}. Letting $\epsilon\rightarrow 0$ in \eqref{fqg327211IIIc} and \eqref{sd327211IIIc}, an application of the lemma gives that $\mathcal{B}_1$ is a quasi-greedy Markushevich basis for $X$, with first quasi-greedy constant
$$
K_{1q}(\mathcal{B}_1)\le 2K_{1q}(\mathcal{B}_2)+1\le
3+4K_{ws}(\tau,\mathcal{B})+16K_{ws}(\tau,\mathcal{B})^2\tau^{-2},
$$
and it is superdemocratic with constant
$$
K_{sd}(\mathcal{B}_1)\le 4K_{sd}(\mathcal{B}_2)\le 32K_{ws}(\tau,\mathcal{B})^2\tau^{-2}.
$$
To finish the proof, let $v\in X$ and suppose there is a subnet $(x_{i_\gamma})_{\gamma}$ such that
$$
x_{i_\gamma}\xrightarrow{w}v.
$$
It is immediate that $x_j'(v)=0$ for all $j\in \N$. Then, as $\mathcal{B}_1$ is a Markushevich basis for $X$ we get that $v-x_0'(v)x_0=0$. This proves that $\overline{\{x_i\}}^{w}_{i\in\N}\subseteq \{x_i\}_{i\in\N}\cup [x_0]$.
\end{proof}

\section{Finite dimensional spaces and branch greedy algorithms.}\label{sectionfiniteandBTA}

In this section, we study the semi-greedy and almost greedy properties - and some weaker versions thereof - in finite dimensional Banach spaces, where   
each biorthogonal system is a greedy Schauder basis. Here, the questions concerning (weak) thresholding or Chebyshev greedy algorithms  focus on the behavior and the relationships of their natural associated constants. We will consider branch semi-greedy and branch almost greedy bases, introduced and studied by Dilworth, Kutzarova, Schlumprecht and Wojtaszczyk in \cite{Dilworth2012}, and extend some of their results. 
\\
Let us present the \emph{branch} versions of  the (weak) thresholding and Chebyshev greedy algorithms. For a fixed weakness parameter $0<\tau< 1$ and a Markushevich basis $(x_i)_{i}$ with seminormalized coordinates, the algorithm is defined as follows. First, set
$$
\mathcal{A}^{\tau}(x):=\{ i\in\N\colon |e_i'(x)|\ge \tau\max\limits_{j\in \N}|e_j'(x)|\},
$$
and let $\mathcal{G}^{\tau}\colon X\setminus \{0\}\rightarrow \N$ be a function with the following
properties:
\begin{enumerate}[\upshape (BG1)]
\item \label{firstbranchgreedy} $\mathcal{G}^{\tau}(x)\in \mathcal{A}^{\tau}(x)$ for every $x\in
    X\setminus \{0\}$.
\item \label{secondbranchgreedy} $\mathcal{G}^{\tau}(\lambda x)=\mathcal{G}^{\tau}(x)$ for all $x\in
    X\setminus \{0\}$ and all $\lambda \in \mathbb{K}\setminus \{0\}$.
\item \label{thirdbranchgreedy} If $\mathcal{A}^{\tau}(x)=\mathcal{A}^{\tau}(y)$ and
    $e_i'(x)=e_i'(y)$ for all $i\in \mathcal{A}^{\tau}(x)$, then
    $\mathcal{G}^{\tau}(x)=\mathcal{G}^{\tau}(y)$.
\end{enumerate}
For each $x\not=0$, this defines a function $\rho_x^{\tau}\colon \{ 1,\dots,|\supp(x)|\}\rightarrow
\N$ if $|\supp(x)|<\infty$ or $\rho_x^{\tau}\colon \N\rightarrow \N$ otherwise, given by
$\rho^{\tau}_{x}(1):=\mathcal{G}^{\tau}(x)$, and for $2\le i\le |\supp(x)|$,
$$
\rho^{\tau}_{x}(i):=\mathcal{G}^{\tau}(x-\sum\limits_{j=1}^{i-1}x_{\varrho_{x}^{\tau}(j)}'(x)x_{\varrho_{x}^{\tau}(j)}).
$$
Similarly, for every $x\in X\setminus \{0\}$ and $m\in\N$, the \emph{$m$-term branch greedy
approximation} to $x$ (with regard to a fixed branch greedy algorithm)  is defined as
$$
\mathcal{G}_m^{\tau}(x):= \sum\limits_{i=1}^{m}x_{\varrho_{x}^{\tau}(i)}'(x)x_{\varrho_{x}^{\tau}(i)},
\label{taugreedyapproximation}
$$
setting $x_{\varrho_{x}^{\tau}(i)}'(x)x_{\varrho_{x}^{\tau}(i)}:=0$ if $i>\max{(\supp(x))}$, and $\mathcal{G}_0^{\tau}(x):=0$.
The idea of choosing a branch associated to a weakness parameter $\tau$ is applied to different concepts (see \cite{Dilworth2015}). 
\begin{definition}\cite[Definition 6.1]{Dilworth2012}\label{definitionbranchalmostgreedy583830}
Let $N\in \N$, and let $E$ be a $N$-dimensional Banach space with a fundamental minimal system
$(x_i')_{1\le i\le N}\subseteq E$. The system is called \emph{branch almost-greedy} with weakness
parameter $0<\tau\le 1$ (BAG($\tau$)) and constant $M$ if, for every $x\in E$ and every $0\le m\le N$,
we have
$$
\|x-\mathcal{G}_m^{\tau}(x)\|\le M\widetilde{\sigma}_{m}(x).
$$
\end{definition}

\begin{definition}\cite[Definition 7.3]{Dilworth2012}
\label{definitionbranchsemigreedy457573}
Let $N\in \N$, and let $E$ be a $N$-dimensional Banach space with fundamental minimal system
$(x_i)_{1\le i\le N}\subseteq E$. The system is called \emph{branch semi-greedy} with weakness
parameter $0<\tau< 1$ (BSG($\tau$)) and constant $M$ if, for every $x\in E$ and every $1\le m\le N$,
there are scalars $(a_i)_{1\le i\le m}$ such that
$$
\|x-\sum\limits_{i=1}^{m}a_ix_{\varrho_{x}^{\tau}(i)}\|\le M\sigma_m(x).
$$
\end{definition}

\begin{remark}Note that if we consider the definition of WAG($\tau$) systems in the finite dimensional context, it is immediate that every BAG($\tau$) system with constant $M$ is also WAG($\tau$) with constant no greater than $M$. The same relation exists between WSG($\tau$) and BSG($\tau$) systems. 
\end{remark}

\begin{remark}\label{remarkgreedybranchgreedy039434} Also, note that the greedy ordering provides a branch greedy algorithm with parameter $\tau$ for every $0<\tau<1$. We can simply define
$$
\mathcal{G}^{\tau}(x):= \rho(x,1)
$$
for all $x\in X$. 
It is easy to check that $\mathcal{G}^{\tau}$ satisfies \ref{firstbranchgreedy}, \ref{secondbranchgreedy}, and \ref{thirdbranchgreedy} for all $0<\tau<1$. 
\end{remark}

Every BAG$(\tau)$ system with constant $M$ has quasi-greedy, democratic and almost greedy constants depending only on $M$ and $\tau$ \cite[Theorem 6.4, Corollary 6.5]{Dilworth2012}. Also, for an almost greedy system, the conditions of Definition~\ref{definitionweaksemigreedy747473} hold for all $x\in X$, $m\in \N$, and \emph{every} weak thresholding set $\W^{\tau}(x,m)$, with $M$ depending only on the first quasi-greedy constant and $\tau$ (\cite[Theorem 7.1]{Dilworth2012}). This implies immediately that it is BSG($\tau$), and that every branch of the algorithm satisfies the  BSG condition.\\
Going in the opposite direction,  that is from the BSG($\tau$) to the almost greedy (or, equivalently, the BAG($\tau$)) property, it was proved in \cite[Theorem 7.7]{Dilworth2012} that the almost greedy constant can be controlled by the BSG($\tau$) constant, $\tau$, the basis constant, and the cotype constant of the space (\cite[Theorem 7.7]{Dilworth2012}. In the same paper, the authors left open the question of whether the BSG($\tau$) property implies in general the BAG($\tau$) property, that is if the constant of the latter can be controlled by that of the former (see the question below \cite[Definition 7.3]{Dilworth2012}). Now we are in a position to answer that question and extend \cite[Theorem 7.7]{Dilworth2012}. \\
First, note that in Example~\ref{exampleseminotalmost493821}, we did not use that the space is infinite dimensional to prove any of the bounds for the constants of the system, except for the lower bound for the WAG($\tau$) constants, for which we used Proposition~\ref{propositionwagag394995}. The proofs of the rest of the bounds hold if we replace $\ell_1$ with $\ell_1^{n}$ for any $n\ge 3$. The system in Example~\ref{exampleseminotalmost493821} has semi-greedy constant no greater than $4$, and in fact, by \ref{semigreedy493821}, all branches of the algorithm satisfy the BSG($\tau$) condition with constant no greater than $4\tau^{-1}$. Hence, \ref{l1493821} and \ref{quasigreedy493821} show that there is no upper bound for the democracy, quasi-greedy or almost greedy constant that depends only on the semi-greedy or the BSG($\tau$) constant and $\tau$. Thus, by \cite[Theorem 6.4]{Dilworth2012} and \cite[Corollary 6.5]{Dilworth2012}, it follows that there is no such upper bound for the BAG($\tau$) constant, either.

Second, it is possible to remove the cotype condition from \cite[Theorem 7.7]{Dilworth2012}, and also extend the result to any WSG($\tau$) system. To do so, next we provide a bound for the second quasi-greedy constant of such systems. For the proof we combine ideas from the proofs of \cite[Theorem 1.10]{Berna2019} and Theorem~\ref{theoremsemigreedyalmostgreedy327211partI} with further arguments that allow us to handle the finite dimensional case.

\begin{theorem}\label{Theoremfinitewsg397311} Let $N\in \N_{>1}$ and $E$ be a $N$-dimensional Banach space with  a WSG($\tau$) basis $(x_i')_{1\le i\le N}$, $0<\tau\le 1$. If $(x_i)_{1\le i\le N}$ has WSG($\tau$) constant
$K_{ws}(\tau)$ and basis constant $K_b$, then $(x_i)_{1\le i\le N}$ is quasi-greedy with second
quasi-greedy constant
$$
K_{2q}\le 5K_b^2K_{ws}(\tau)+6 K_b^3K_{ws}(\tau)^2\tau^{-2}.
$$
\end{theorem}

\begin{proof}

Let $N_1:= \floor{\frac{N+1}{2}}$, and consider the finite sets $A_1:=\{j \in \N\colon 1\le j\le
N_1\}$ and $A_2:=\{j\in \N\colon N_1<j\le N\}$. Now, for all $x\in E$ and $i=1,2$ define the  projection operators
$$
P_{i}(x):=\sum\limits_{j\in A_i}x_j'(x)x_j.
$$
Fix $x\in E$ and $1\le m\le N$, assuming without loss of generality that $\G_m(x)\not=x$ (else, there is
nothing to prove). Set
$$
m_1:=|A_1\cap \mathcal{GS}_{m}(x)|\quad \text{and}\quad
m_2:=|A_2\cap \mathcal{GS}_{m}(x)|.
$$
Note that
$$
\mathcal{G}_{m}(x)=\mathcal{G}_{m_1}(P_1(x))+\mathcal{G}_{m_2}(P_2(x)).
$$
Thus,
\begin{equation}
\|x-\mathcal{G}_{m}(x)\|\le
\|P_1(x)-\mathcal{G}_{m_1}(P_1(x))\|+\|P_2(x)-\mathcal{G}_{m_2}(P_2(x))\|.\label{splitting397311}
\end{equation}
Let us consider first the case in which $m_1\not=0$ and $m_2\not=0$. Since
$x\not=\mathcal{G}_{m}(x)$, it follows that $x_{\rho(P_i(x),m_i)}'(P_i(x))\not=0$ for $1\le i\le 2$.
Fix $0<\xi<1$, and let
\begin{align*}
y_1:=&\tau^{-1}(1+\xi)|x_{\rho(P_1(x),m_1)}'(P_1(x))|\sum\limits_{j=N_1+1}^{N_1+m_1}x_j;\\
y_2:=&\tau^{-1}(1+\xi)|x_{\rho(P_2(x),m_2)}'(P_2(x))|\sum\limits_{j=1}^{m_2}x_j.
\end{align*}
Note that for any $N_1< j \le (N_1+m_1)$ and $1\le i\le N_1$ or $(N_1+m_1)<i\le N$, we have
$$
\tau x_j'(y_1)=(1+\xi)|x_{\rho(P_1(x),m_1)}'(P_1(x))|>|x_i'(P_1(x)-\mathcal{G}_{m_1}(P_1(x)))|.
$$
Hence, the only $m_1$-weak thresholding set for
$$
P_1(x)-\mathcal{G}_{m_1}(P_1(x))+y_1
$$
with weakness parameter $\tau$ is the set $\{j\colon N_1< j\le N_1+m_1\}$. Similarly, the only $m_2$-weak thresholding set for
$$
P_2(x)-\mathcal{G}_{m_2}(P_2(x))+y_2
$$
with weakness parameter $\tau$ is the set $\{j\colon 1\le j\le m_2\}$. Let $w_1$ and $w_2$ be an $m_1$-term and an $m_2$-term
Chebyshev $\tau$-greedy approximant for $P_1(x)-\mathcal{G}_{m_1}(P_1(x))+y_1$ and
$P_2(x)-\mathcal{G}_{m_2}(P_2(x))+y_2$, respectively. Considering that $\|P_1\|\le K_b$ and
$\|P_2\|\le 1+K_b$, we deduce that
\begin{align*}
 \|P_1(x)-\mathcal{G}_{m_1}(P_1(x))\|\le& K_b \|P_1(x)-\mathcal{G}_{m_1}(P_1(x))+y_1-w_1\|\nonumber\\
 \le& K_bK_{ws}(\tau)\sigma_{m_1}(P_1(x)-\mathcal{G}_{m_1}(P_1(x))+y_1)\nonumber\\
 \le& K_bK_{ws}(\tau)\|P_1(x)\|+K_bK_{ws}(\tau)\|y_1\|\nonumber\\
 \le& K_b^2K_{ws}(\tau)\|x\|+K_bK_{ws}(\tau)\|y_1\|.
\end{align*}
Analogously, we get that
$$
\|P_2(x)-\mathcal{G}_{m_2}(P_2(x))\|\le(1+K_b)^2K_{ws}(\tau)\|x\|+(1+K_b)K_{ws}(\tau)\|y_2\|.
$$
Reasoning as before, we see that any weak thresholding set of cardinality $m_1$ for
$$
P_1(x)+\tau^{2}(1-\xi)(1+\xi)^{-1}y_1
$$
is contained in $\{1\le j\le N_1\}$. So taking $u_1$ an $m_1$-term Chebyshev $\tau$-greedy approximant for $P_1(x)+\tau^{2}(1-\xi)(1+\xi)^{-1}y_1$, we deduce that

\begin{align*}
\|y_1\|=& \tau^{-2}(1-\xi)^{-1}(1+\xi)\|\tau^2 (1-\xi)(1+\xi)^{-1}y_1\|\nonumber\\
\le& \tau^{-2}(1-\xi)^{-1}(1+\xi)(1+K_b)\|P_1(x)-u_1+\tau^{2}(1-\xi)(1+\xi)^{-1}y_1\|\nonumber\\
\le&
\tau^{-2}(1-\xi)^{-1}(1+\xi)(1+K_b)K_{ws}(\tau)\sigma_{m_1}(P_1(x)+\tau^{2}(1-\xi)(1+\xi)^{-1}y_1)\nonumber\\
\le&\tau^{-2}(1-\xi)^{-1}(1+\xi)(1+K_b)K_{ws}(\tau)\|P_1(x)\|\nonumber\\
\le& \tau^{-2}(1-\xi)^{-1}(1+\xi)(1+K_b)K_bK_{ws}(\tau)\|x\|.
\end{align*}
Similarly, we obtain
$$
\|y_2\|\le \tau^{-2}(1-\xi)^{-1}(1+\xi)(1+K_b)K_bK_{ws}(\tau)\|x\|.
$$
From the above estimations, and letting $\xi\rightarrow 0$,  we deduce that

\begin{align}
 \|P_1(x)-\mathcal{G}_{m_1}(P_1(x))\|\le&
 (K_b^2K_{ws}(\tau)+(1+K_b)K_b^2K_{ws}^2(\tau)\tau^{-2})\|x\|;\label{5bound397311}\\
  \|P_2(x)-\mathcal{G}_{m_2}(P_2(x))\|\le&
  ((1+K_b)^2K_{ws}(\tau)+(1+K_b)^2K_bK_{ws}^2(\tau)\tau^{-2})\|x\|.\label{6bound397311}
\end{align}
Now suppose that $m_1=0$. Then $\G_{m_1}(P_1(x))=0$, so  \eqref{5bound397311} is clear, and we can obtain
\eqref{6bound397311} by the same argument as before because $m_2\le N_1$.
Finally, assume $m_2=0$. Then, \eqref{6bound397311} is clear. Now, if $m_1<N_1$, we apply the
same argument as before to obtain \eqref{5bound397311}. On the other hand, if  $m_1=N_1$, then
$\mathcal{G}_{m_1}(P_1(x))=P_1(x)$, so \eqref{5bound397311} is immediate. \\
To finish the proof, from  \eqref{splitting397311}, \eqref{5bound397311} and \eqref{6bound397311} we
infer that
$$
\|x-\mathcal{G}_{m}(x)\| \le (5K_b^2K_{ws}(\tau)+6 K_b^3K_{ws}(\tau)^2\tau^{-2})\|x\|,
$$
from where the upper bound for $K_{2q}$ is obtained.
\end{proof}
Note that in \cite[Theorem 7.4]{Dilworth2012} it was proved that any system $(x_i)_{1\le i\le N}$ that is BSG($\tau$) is also superdemocratic with constant depending only on the basis constant, $\tau$, and the BSG($\tau$) constant. A careful look at the proof shows that it is also valid for WSG($\tau$) systems. Also, the bounds in Theorem~\ref{Theoremalmostgreedydemocraticquasigreedy} are extracted from the proof of \cite[Theorem 3.3]{Dilworth2003} (with minor modifications for complex scalars), which is valid for finite dimensional spaces. Combining these results with Theorem~\ref{Theoremfinitewsg397311}, we obtain the following extension of \cite[Theorem 7.7]{Dilworth2012}. 

\begin{theorem}\label{extension} Let $N\in \N$ and let $E$ be an $N$-dimensional Banach space. Let $(x_i)_{1\le i\le N}\subset E$ be a WSG($\tau$) system for $E$, $0<\tau \le 1$, with constant $K_{ws}$ and basis constant $K_b$. Then, $(x_i)_{1\le i\le N}\subset E$ has almost greedy constant depending only on $K_{ws}$, $\tau$ and $K_b$.
\end{theorem}

Finally, we note that the branch thresholding algorithm can be and has been considered in infinite dimensional spaces as well. Indeed, in \cite{Dilworth2012}, the authors do so and prove that every weakly null semi-normalized branch quasi-greedy basic sequence has a quasi-greedy subsequence. If we extend the definitions of branch semi-greedy and branch almost greedy systems to the infinite dimensional context in the natural manner, it is immediate from the definitions that every semi-greedy system is branch semi-greedy, every branch semi-greedy system is weak semi-greedy, and the corresponding implications hold for the almost greedy case. Therefore, BSG($\tau$) and BAG($\tau$) Markushevich bases can be added to the equivalences in Corollary~\ref{corollaryequivalences}.

\end{document}